\documentclass{article}
\usepackage{amsfonts}
\usepackage{amsthm}			
\usepackage{amssymb}
\usepackage{amsmath}
\usepackage{hyperref}
\usepackage{tikz}
\usetikzlibrary{shapes.geometric,arrows}
\usetikzlibrary{calc,patterns}
\usepackage{ifthen}
\theoremstyle{plain}
\usepackage{bbm}
\usepackage{float}
\usepackage{authblk}
\usepackage{mathdots}
\usepackage{mathtools}
\usepackage{euscript}
\usepackage{dsfont}
\usepackage{rotating}
\usepackage{dsfont}
\usepackage{fullpage}
\theoremstyle{definition}
\newtheorem{thm}{Theorem}[section]
\newtheorem{lemma}[thm]{Lemma}
\newtheorem{proposition}[thm]{Proposition}
\newtheorem{cor}[thm]{Corollary}

\newtheorem{definition}[thm]{Definition}

\newtheorem{exmp}[thm]{Example}

\theoremstyle{remark}
\newtheorem*{rem}{Remark}

\newcommand{\rt}[1]{\Phi^{+}(#1)}

\newcommand{\tpp}[2]{\node at (#1*1+.25,#2*1-.25) {1};\node at (#1*1+.75,#2*1-.75) {-1};}
\newcommand{\tpm}[2]{
\node at (#1*1+.25,#2*1-.25) {-1};
\node at (#1*1+.75,#2*1-.75) {1};}
\newcommand{\tp}[2]{\node at (#1*1+.25,#2*1-.25) {1};}
\newcommand{\tm}[2]{\node at (#1*1+.75,#2*1-.75) {1};}
\newcommand{\precdot}
{\prec\mathrel{\mkern-5mu}\mathrel{\cdot}}
\DeclareMathOperator\rk{rk}

\newcommand{\Tout}[1]{\mathcal{T}_{#1}^-}

\title{Dimensions of toggleability spaces}

\begin{document}

\author[*]{Ben Adenbaum}

\author[**]{Spencer Daugherty}

\author[$\dagger$]{Nicholas Mayers}

\affil[*]{Department of Mathematics and Statistics, Villanova University, Villanova, PA, 18095}

\affil[**]{Department of Mathematics, University of Colorado Boulder, Boulder, CO, 80309}

\affil[$\dagger$]{Department of Mathematics, Kennesaw State University, Kennesaw, GA, 30144}

\maketitle

\begin{abstract}

   We establish a conjecture of Defant, Hopkins, Poznanovi\'{c}, and Propp concerning the dimensions of toggleability spaces for products of chains, shifted staircases, type-A root posets, and type-B posets. Generalizing this result, we show that for a larger family of posets defined by restricted diagrams, the dimensions of toggleability spaces are equal to the rank of the poset plus one. As part of our approach, we build upon the technique of rook statistics introduced by Chan, Haddadan, Hopkins, and Moci.
\end{abstract}
\section{Introduction}\label{sec:intro}

Rowmotion, one of the most studied actions in dynamical algebraic combinatorics, is an action on the distributive lattice of order ideals of a finite poset which sends an order ideal $I$ of a poset $\mathcal{P}$ to the order ideal generated by the minimal elements of $\mathcal{P} \setminus I$. A large focus of research on rowmotion is concerned with identifying statistics that exhibit the property of homomesy under this action. A statistic on the order ideals of a poset is homomesic under rowmotion if its average value across the orbits is constant. Two popular statistics to study that are often homomesic are the order ideal and antichain cardinality statistics, which give the size of the input order ideal or its maximal antichain, respectively. 

Various authors have approached the study of homomesic statistics under rowmotion using toggleability statistics \cite{defant2021homomesy, ElderToggICS, mertin2024toggleability, striker2015toggle, striker2018rowmotion}. For a poset $\mathcal{P}$, $p\in\mathcal{P}$, and an order ideal $I$ of $\mathcal{P}$, the toggleability statistic $\mathcal{T}_p$ captures whether an element $p$ is maximal in $I$, minimal in $\mathcal{P} \setminus I$, or neither. The order ideal indicator function $\mathbbm{1}_p$ indicates whether an element $p$ is contained in $I$, while the antichain indicator function $\mathcal{T}^-_p$ does the same for the antichain associated with $I$. 

In \cite{defant2021homomesy}, the authors consider the intersections of the span of toggleability statistics of a poset $\mathcal{P}$ along with a constant with $\mathrm{Span}_{\mathbb{R}}(\mathbbm{1}_p~|~p\in\mathcal{P})$ and $\mathrm{Span}_{\mathbb{R}}(\mathcal{T}^-_p~|~p\in\mathcal{P})$, denoted $I_T(\mathcal{P})$ and $A_T(\mathcal{P})$, respectively. Every function in the spaces $I_T(\mathcal{P})$ and $A_T(\mathcal{P})$ is homomesic due to properties of toggleability statistics \cite{striker2015toggle}. Restricting attention to products of two chains $[m]\times [n]$, shifted staircases $([n]\times [n])\backslash \mathfrak{S}_2$, the type $A$ root posets $\rt{A_n}$, and the type $B$ posets $\rt{B_n}$, the authors of \cite{defant2021homomesy} conjectured values for both $\dim I_T(\mathcal{P})$ and $\dim A_T(\mathcal{P})$. Here, we provide a proof of their conjecture. Specifically, letting $\mathrm{rk}(\mathcal{P})$ denote the rank of a poset $\mathcal{P}$, we prove the following.

\setcounter{section}{2} 
\setcounter{thm}{2}
\begin{thm}~
    \begin{enumerate}
    \item[$(a)$] If $\mathcal{P}=[m]\times [n]$ for $m,n\ge 2$, then $\dim I_T(\mathcal{P})=\dim A_T(\mathcal{P})=\mathrm{rk}(\mathcal{P})+1=n+m-1$.
    \item[$(b)$] If $\mathcal{P}=([n]\times[n])\backslash \mathfrak{S}_2$ for $n\ge 2$, then $\dim I_T(\mathcal{P})=\dim A_T(\mathcal{P})=\mathrm{rk}(\mathcal{P})+1=2n-1$.
    \item[$(c)$] If $\mathcal{P}=\Phi^+(A_n)$ for $n\ge 2$, then $\dim I_T(\mathcal{P})=\dim A_T(\mathcal{P})=\mathrm{rk}(\mathcal{P})+1=n$.
    \item[$(d)$] If $\mathcal{P}=\Phi^+(B_n)$ for $n\ge 2$, then $\dim I_T(\mathcal{P})=\dim A_T(\mathcal{P})=\mathrm{rk}(\mathcal{P})+1=2n-1$.
\end{enumerate}
\end{thm}
\setcounter{section}{1}
\setcounter{thm}{0}

\noindent
Moreover, we provide bases for each of the toggleability spaces $I_T(\mathcal{P})$ considered and, in proving Theorem~\ref{thm:main}, establish a number of results concerning $\dim I_T(\mathcal{P})$ and $\dim A_T(\mathcal{P})$ for more general posets $\mathcal{P}$. Through such results, we are led to two noteworthy extensions of Theorem~\ref{thm:main}. First, given an integer partition $\lambda=[\lambda_1,\hdots,\lambda_\ell]$, in Section~\ref{sec:ip} we define a poset $$\mathcal{P}(\lambda)=\bigcup_{i=1}^\ell\{(i,j)~|~1\le j\le \lambda_i\}$$ whose elements are ordered lexicographically. For such a poset $\mathcal{P}=\mathcal{P}(\lambda)$, we show that both $\dim I_T(\mathcal{P})$ and $\dim A_T(\mathcal{P})$ have a particularly pleasing form in terms of the border strip of $\lambda$.

\setcounter{section}{3} 
\setcounter{thm}{15}
\begin{thm}
    For an integer partition $\lambda$, letting $N$ denote the number of cells and $C$ the number of corner cells in the border strip of $\lambda$, we have that $\dim I_T(\mathcal{P})=\dim A_T(\mathcal{P})=N-C$.
\end{thm}
\setcounter{section}{1}
\setcounter{thm}{0}

\noindent
For the second extension, we consider distributive lattices of order ideals of finite, width-two posets. 

\setcounter{section}{4} 
\setcounter{thm}{1}
\begin{thm}
    Let $\mathcal{P}$ be the distributive lattice of order ideals of a finite, width-two poset. Then \[\dim A_T(\mathcal{P})=\dim I_T(\mathcal{P})=\rk(\mathcal{P})+1.\]
\end{thm}
\setcounter{section}{1}
\setcounter{thm}{0}

\noindent
Our approach to Theorem~\ref{thm:dim} builds upon the technique of rook statistics introduced by Chan, Haddadan, Hopkins, and Moci in~\cite{chan2017expected}.

The remainder of the paper is organized as follows. In Section~\ref{sec:background}, we cover the requisite background from the theory of posets. Following this, the proof of Theorem~\ref{thm:main} is split between Sections~\ref{sec:conj1} and~\ref{sec:antichain}. In Section~\ref{sec:conj1} we focus on the space $I_T(\mathcal{P})$, while in Section~\ref{sec:antichain} we turn our focus to $A_T(\mathcal{P})$. 

\subsection*{Acknowledgments}
The authors are very grateful to Sam Hopkins for  helpful discussions. We would also like to thank the organizers of the 2024 AMS Math Research Community in Algebraic Combinatorics, where this collaboration began.

\section{Background}\label{sec:background}

In this section, we cover the required background from the theory of posets. For more details, see~\cite[Chapter 3]{stanley2012enumerative}. Recall that a \textbf{finite poset} $(\mathcal{P},\preceq_{\mathcal{P}})$ consists of a finite set $\mathcal{P}$ along with a binary relation $\preceq$ between the elements of $\mathcal{P}$ which is reflexive, antisymmetric, and transitive. When no confusion will arise, we denote $(\mathcal{P},\preceq_{\mathcal{P}})$ simply by $\mathcal{P}$, and $\preceq_{\mathcal{P}}$ by $\preceq$. Ongoing, we let $\le$ denote the relation corresponding to the natural ordering of $\mathbb{Z}$. Two posets $\mathcal{P}$ and $\mathcal{Q}$ are called \textbf{isomorphic} if there exists an order-preserving bijection $\phi:\mathcal{P}\to\mathcal{Q}$ whose inverse is also order preserving.

For a poset $\mathcal{P}$ and $x,y\in\mathcal{P}$, if $x\preceq y$ and $x\neq y$, then we write $x\prec y$. In the case where $x\prec y$ and there exists no $z\in\mathcal{P}$ satisfying $x\prec z\prec y$, we refer to $x\prec y$ as a \textbf{cover relation} and write $x\precdot y$. We refer to an element $x\in\mathcal{P}$ as \textbf{minimal} (resp., \textbf{maximal}) if there exists no $y\in\mathcal{P}$ satisfying $y\prec x$ (resp., $x\prec y$). All of the posets of interest in this paper can be defined, or visually represented, using ``diagrams" as in \cite{defant2021homomesy}. Here, a \textbf{diagram} is an array of finitely many cells in $\mathbb{N}\times\mathbb{N}$. For an example, see Figure~\ref{fig:poset} (a). Given a diagram $D$, one can associate a poset $\mathcal{P}(D)$ whose elements are the cells of $D$ and where for $p_1,p_2\in\mathcal{P}$, we have $p_1\preceq p_2$ provided that the box associated with $p_2$ lies weakly southeast of the box associated with $p_1$. For a poset $\mathcal{P}$, if there exists a diagram $D$ for which $\mathcal{P}$ is isomorphic to $\mathcal{P}(D)$, then we say that $\mathcal{P}$ is \textbf{defined} by the diagram $D$. An example of a poset $\mathcal{P}$ defined by a diagram is provided in Example~\ref{ex:poset}. When indexing cells in diagrams, we use matrix coordinates. We refer to a diagram as
\begin{itemize}
    \item \textbf{connected} if there exists no pairs of cells which cannot be connected by a path of pairwise edge adjacent cells;
    \item \textbf{row convex} (resp., \textbf{column convex}) if between any two cells in the same row (resp., column), all cells between are contained as well; and
    \item \textbf{simply connected} if there exist no holes in the diagram, i.e., there exists no $(i,j)$ such that cells are contained in positions $(i-1,j),(i+1,j),(i,j-1),(i,j+1)$ and no cell is contained in the $(i,j)$ position.
\end{itemize}
See Example~\ref{ex:poset} for an illustration of these properties of diagrams.

\begin{exmp}\label{ex:poset}
    Let $\mathcal{P}=\{a,b,c,d,e,f\}$ be the poset defined by the cover relations $a\precdot c$; $b\precdot c\precdot d\precdot f$; and $c\precdot e\precdot f$. Then $\mathcal{P}$ is defined by the diagram $D$ illustrated in Figure~\ref{fig:poset} $(a)$. The cells of $D$ have been labeled by the corresponding elements of $\mathcal{P}$ to help make the identification clear. Note that the diagram $D$ is both connected and simply connected, as well as row and column convex. On the other hand, the diagram $T$ of Figure~\ref{fig:poset} $(b)$ is not row convex as a consequence of the two cells marked with a $\ast$. Similarly, $T$ is not column convex as a consequence of the two cells labeled $\dagger$. Finally, if the cell labeled by $\bigtimes$ in $T$ did not belong to $T$, then $T$ would not be simply connected.
    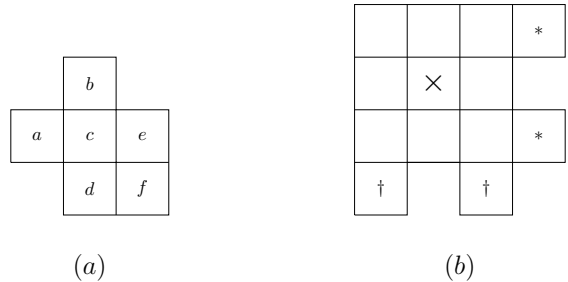
\begin{figure}[H]
        \centering
        $$\scalebox{0.7}{\begin{tikzpicture}
    \def\Node{\node [circle, fill, inner sep=1.5pt]}
    \draw (0,0)--(1,0)--(1,-1)--(3,-1)--(3,1)--(2,1)--(2,2)--(1,2)--(1,1)--(0,1)--(0,0);
    \draw (3,0)--(2,0)--(2,1)--(1,1)--(1,0)--(2,0)--(2,-1);
    \node at (0.5,0.5) {$a$};
    \node at (1.5, 1.5) {$b$};
    \node at (1.5, 0.5) {$c$};
    \node at (1.5, -0.5) {$d$};
    \node at (2.5, 0.5) {$e$};
    \node at (2.5, -0.5) {$f$};
    \node at (1.5, -2) {\Large $(a)$};
\end{tikzpicture}}\quad\quad\quad\quad\quad\quad\quad \scalebox{0.7}{\begin{tikzpicture}
    \def\Node{\node [circle, fill, inner sep=1.5pt]}
    \draw (0,0)--(1,0)--(1,1)--(2,1)--(2,0)--(3,0)--(3,1)--(4,1)--(4,2)--(3,2)--(3,3)--(4,3)--(4,4)--(0,4)--(0,0);
    \draw (1,1)--(1,4);
    \draw (2,1)--(2,4);
    \draw (3,1)--(3,4);
    \draw (0,1)--(3,1);
    \draw (0,2)--(3,2);
    \draw (0,3)--(3,3);
    \node at (3.5,1.5) {$\ast$};
    \node at (3.5,3.5) {$\ast$};
    \node at (0.5,0.5) {$\dagger$};
    \node at (2.5,0.5) {$\dagger$};
    \node at (1.5,2.5) {$\bigtimes$};
    \node at (2,-1) {\Large $(b)$};
\end{tikzpicture}}$$
        \caption{Diagrams defining a poset}
        \label{fig:poset}
    \end{figure}
\end{exmp}

\noindent
For a simply connected diagram $D$, an \textbf{outward corner} of $D$ associated to a cell $(i,j)\in D$ is a corner on the boundary of $D$ of the form $EN$ (\scalebox{0.2}{$\begin{tikzpicture}
    \draw (0,0)--(1,0)--(1,1);
\end{tikzpicture}$}) strictly to the northwest of $(i,j)$ or a corner on the boundary of $D$ of the form $NE$ (\scalebox{0.2}{$\begin{tikzpicture}
    \draw (0,0)--(0,1)--(1,1);
\end{tikzpicture}$}) strictly to the southeast of $(i,j)$. A diagram is said to have \textbf{no outward corners} if no cell in $D$ has an outward corner.

\begin{figure}[H]
    \centering
        \scalebox{0.6}{\begin{tikzpicture}

\draw[step=1.0,black,thin] (0,0) grid (6,6);
\draw[-, ultra thick, blue] (1,6)--(3,6)--(3,3)--(4,3)--(4,1)--(6,1)--(6,0)--(2,0)--(2,1)--(0,1)--(0,5)-- (1,5)--(1,6);

\draw[-,ultra thick, orange] (0,5) -- (1,5)--(1,6);
\filldraw [black] (2.5,3.5) circle (3pt);  
 \begin{scope}[shift = {(8, 0)}]
    \draw[step=1.0,black,thin] (0,0) grid (4,3);
         \draw[-, blue, ultra thick] (0,3) -- (0,1)--(2,1)--(2,0)--(4,0)--(4,2)--(2,2)--(2,3)--(0,3);
 \end{scope}
\end{tikzpicture}}
    \caption{A diagram with an outward corner, colored in orange, associated to the cell (3,3) (left) and a diagram with no outward corners (right)}
    \label{fig:cell_diagram_with_outward_corner}
\end{figure}
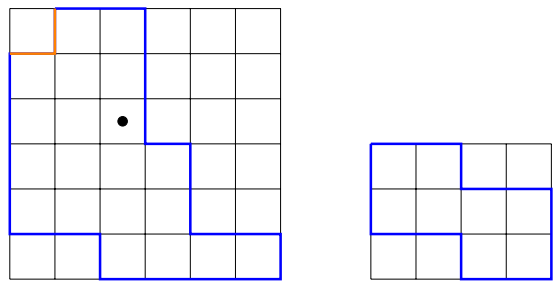

In this paper, we mainly focus on four families of posets. For $m,n\ge 2$, we denote by $[m]\times [n]$ the poset defined by the full $m\times n$ rectangular diagram. More formally, $[m]\times [n]$ is the poset with elements $(i,j)$ for $1\le i\le m$ and $1\le j\le n$, and cover relations
\begin{itemize}
    \item $(i,j)\precdot (i+1,j)$ for $1\le i<m$ and $1\le j\le n$ and
    \item $(i,j)\precdot (i,j+1)$ for $1\le i\le m$ and $1\le j<n$.
\end{itemize}
For example, the diagram $D$ defining $\mathcal{P}=[3]\times[4]$ is illustrated in Figure~\ref{fig:theposets} (a), where the cells are labeled by the associated elements of $\mathcal{P}$. Given a poset $\mathcal{P}$ and a subset $S\subset \mathcal{P}$, the \textbf{induced subposet generated by $S$} is the poset $\mathcal{P}_S$ on $S$, where $i\preceq_{\mathcal{P}_S}j$ if and only if $i\preceq_{\mathcal{P}}j$. The remaining three families of posets of interest can all be viewed as certain induced subposets of $[n]\times [n]$ for $n\ge 2$:
\begin{itemize}
    \item $([n]\times [n])\backslash\mathfrak{S}_2$ for $n\ge 2$ denotes the poset defined by the diagram consisting of all cells in the full $n\times n$ square diagram lying on or above the diagonal of the square. More formally, $([n]\times[n])\backslash \mathfrak{S}_2$ for $n\ge 2$, is the induced subposet of $[n]\times[n]$ with elements $(i,j)$ for $1\le i\le j\le n$. For example, the diagram $D$ defining $\mathcal{P}=([4]\times[4])\backslash \mathfrak{S}_2$ is illustrated in Figure~\ref{fig:theposets} (b), where the cells are labeled by the associated elements of $\mathcal{P}$.

    \item $\Phi^+(A_n)$ for $n\ge 2$, denotes the posets defined by the diagram consisting of all cells in the full $n\times n$ square diagram lying on or below the anti-diagonal of the square. More formally, $\rt{A_n}$ for $n\ge 2$, is the induced subposet of $[n]\times[n]$ with elements $(i,j)$ for $1\le i, j\le n$ satisfying $i+j \geq n+1$. For example, the diagram $D$ defining $\mathcal{P}=\rt{A_5}$ is illustrated in Figure~\ref{fig:theposets} (c), where the cells are labeled by the associated elements of $\mathcal{P}$.
    
    \item $\Phi^+(B_n)$ for $n\ge 2$, denotes the poset defined by the diagram consisting of all cells in the full $n\times n$ square diagram lying on or below both the diagonal and anti-diagonal of the square. More formally, $\rt{B_n}$ for $n\ge 2$, is the induced subposet of $[n]\times[n]$ with elements $(i,j)$ for $1 \leq i \leq j \leq 2n-1$ satisfying $i+j \geq 2n$. For example, the diagram $D$ defining $\mathcal{P}=\rt{B_3}$ is illustrated in Figure~\ref{fig:theposets} (d), where the cells are labeled by the associated elements of $\mathcal{P}$.
\end{itemize}

\begin{figure}[H]
    \centering
    $$\scalebox{0.7}{\begin{tikzpicture}
    \def\Node{\node [circle, fill, inner sep=1.5pt]}
    \draw (0,0)--(4,0)--(4,3)--(0,3)--(0,0);
    \draw (0,1)--(4,1);
    \draw (0,2)--(4,2);
    \draw (1,0)--(1,3);
    \draw (2,0)--(2,3);
    \draw (3,0)--(3,3);
    \node at (0.5, 2.5) {$(1,1)$};
    \node at (0.5, 1.5) {$(2,1)$};
    \node at (0.5, 0.5) {$(3,1)$};
    \node at (1.5, 2.5) {$(1,2)$};
    \node at (1.5, 1.5) {$(2,2)$};
    \node at (1.5, 0.5) {$(3,2)$};
    \node at (2.5, 2.5) {$(1,3)$};
    \node at (2.5, 1.5) {$(2,3)$};
    \node at (2.5, 0.5) {$(3,3)$};
    \node at (3.5, 2.5) {$(1,4)$};
    \node at (3.5, 1.5) {$(2,4)$};
    \node at (3.5, 0.5) {$(3,4)$};
    \node at (2,-1) {\Large $(a)$};
\end{tikzpicture}}\quad\quad \scalebox{0.7}{\begin{tikzpicture}
    \draw (-2,4)--(2,4)--(2,0)--(1,0)--(1,1)--(0,1)--(0,2)--(-1,2)--(-1,3)--(-2,3)--(-2,4);
    \draw (-2,3)--(2,3);
    \draw (-1,2)--(2,2);
    \draw (0,1)--(2,1);
    \draw (1,1)--(1,4);
    \draw (0,2)--(0,4);
    \draw (-1,3)--(-1,4);
    \node at (-1.5, 3.5) {$(1,1)$};
    \node at (-0.5, 3.5) {$(1,2)$};
    \node at (0.5, 3.5) {$(1,3)$};
    \node at (1.5, 3.5) {$(1,4)$};
    \node at (-0.5, 2.5) {$(2,2)$};
    \node at (0.5, 2.5) {$(2,3)$};
    \node at (1.5, 2.5) {$(2,4)$};
    \node at (0.5, 1.5) {$(3,3)$};
    \node at (1.5, 1.5) {$(3,4)$};
    \node at (1.5, 0.5) {$(4,4)$};
    \node at (0,-1) {\Large $(b)$};
\end{tikzpicture}}\quad\quad\scalebox{0.7}{
\begin{tikzpicture}
    \draw (0,0)--(5,0)--(5,5)--(4,5)--(4,4)--(3,4)--(3,3)--(2,3)--(2,2)--(1,2)--(1,1)--(0,1)--(0,0);
    \draw (1,1)--(5,1);
    \draw (2,2)--(5,2);
    \draw (3,3)--(5,3);
    \draw (4,4)--(5,4);
    \draw (1,0)--(1,1);
    \draw (2,2)--(2,0);
    \draw (3,3)--(3,0);
    \draw (4,4)--(4,0);
    \node at (4.5,4.5) {$(1,5)$};
    \node at (4.5,3.5) {$(2,5)$};
    \node at (4.5,2.5) {$(3,5)$};
    \node at (4.5,1.5) {$(4,5)$};
    \node at (4.5,0.5) {$(5,5)$};
    \node at (3.5,3.5) {$(2,4)$};
    \node at (3.5,2.5) {$(3,4)$};
    \node at (3.5,1.5) {$(4,4)$};
    \node at (3.5,0.5) {$(5,4)$};
    \node at (2.5,2.5) {$(3,3)$};
    \node at (2.5,1.5) {$(4,3)$};
    \node at (2.5,0.5) {$(5,3)$};
    \node at (1.5,1.5) {$(4,2)$};
    \node at (1.5,0.5) {$(5,2)$};
    \node at (0.5,0.5) {$(5,1)$};
    \node at (2.5,-1) {\Large $(c)$};
\end{tikzpicture}}\quad\quad\scalebox{0.7}{
\begin{tikzpicture}
    \draw (0,0)--(5,0)--(5,1)--(4,1)--(4,2)--(3,2)--(3,3)--(2,3)--(2,2)--(1,2)--(1,1)--(0,1)--(0,0);
    \draw (1,1)--(4,1);
    \draw (2,2)--(3,2);
    \draw (1,0)--(1,1);
    \draw (2,0)--(2,2);
    \draw (3,0)--(3,2);
    \draw (4,0)--(4,1);
    \node at (0.5,0.5) {$(5,1)$};
    \node at (1.5, 0.5) {$(5,2)$};
    \node at (2.5, 0.5) {$(5,3)$};
    \node at (3.5, 0.5) {$(5,4)$};
    \node at (4.5, 0.5) {$(5,5)$};
    \node at (1.5, 1.5) {$(4,2)$};
    \node at (2.5, 1.5) {$(4,3)$};
    \node at (3.5, 1.5) {$(4,4)$};
    \node at (2.5, 2.5) {$(3,3)$};
    \node at (2.5,-1) {\Large $(d)$};
\end{tikzpicture}}$$
    \caption{Diagrams defining (a) $[3]\times [4]$, (b) $([4]\times[4])\backslash\mathfrak{S}_2$, (c) $\rt{A_5}$, and (d) $\rt{B_3}$}
    \label{fig:theposets}
\end{figure}
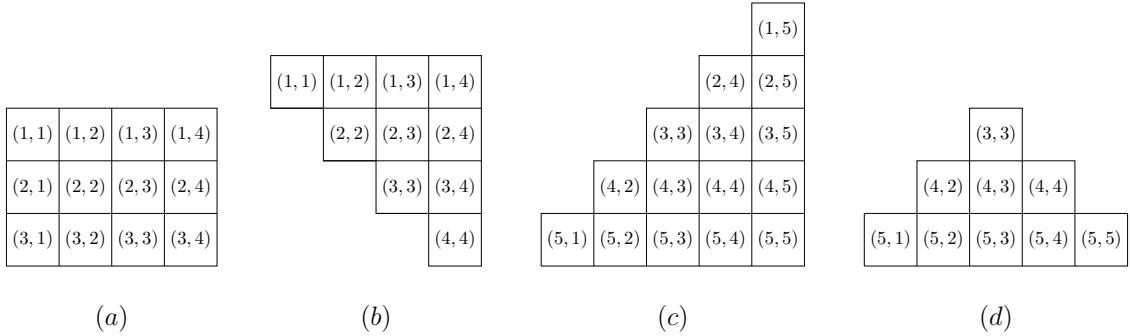

Moving forward, certain subsets of a poset $\mathcal{P}$ play an important role. First, we let $\mathrm{Min}(\mathcal{P})$ (resp., $\mathrm{Max}(\mathcal{P})$) denote the collection of minimal (resp., maximal) elements of $\mathcal{P}$. Next, an \textbf{antichain} of a poset $\mathcal{P}$ is a set of elements in $P$, none of which are comparable to each other. At the other end of the spectrum, a \textbf{chain} is a totally ordered subset of a poset. We define the \textbf{rank} of a poset $\mathcal{P}$ to be one less than the maximal cardinality of a chain in $\mathcal{P}$, denoted $\mathrm{rk}(\mathcal{P})$. Similarly, the \textbf{width} of a finite poset $\mathcal{P}$ is the maximum cardinality of an antichain of $\mathcal{P}$. Finally, an \textbf{order ideal} of a poset $\mathcal{P}$ is a set $I \subseteq \mathcal{P}$ such that if $x\in I$ and $y\prec x$ for $y\in\mathcal{P}$, then $y\in I$. We denote the set of all order ideals of $\mathcal{P}$ by $\mathcal{J}(\mathcal{P})$. Note that an order ideal $I\in\mathcal{J}(\mathcal{P})$ is uniquely identified by the antichain corresponding to its collection of maximal elements. In fact, this identification provides a bijection between the set of order ideals and the set of antichains of a poset $\mathcal{P}$. Ongoing, we denote the order ideal $I\in \mathcal{J}(\mathcal{P})$ with maximal elements $p_1,p_2,\hdots,p_n\in\mathcal{P}$ by $I(p_1,p_2,\hdots,p_n)$.

With respect to posets $\mathcal{P}$ belonging to one of the four families defined above, our interest lies in two associated vector spaces which arise in the study of the dynamics of $\mathcal{J}(\mathcal{P})$ under the action of rowmotion. Letting $\mathcal{P}\backslash S$ denote the poset induced by the elements of $\mathcal{P}$ not contained in $I$, recall from Section~\ref{sec:intro} that, if $I \in \mathcal{J}(\mathcal{P})$, then the \textbf{rowmotion} map $\mathrm{Row}(I)$ yields the order ideal generated by the elements of $\mathrm{Min}(P\setminus I)$. Applying the rowmotion map repeatedly until the original ideal is obtained creates an orbit. In this paper, our focus will be statistics on order ideals that exhibit ``nice" behavior under rowmotion. In a more general setting, letting $\mathrm{st}:S\to \mathbb{Z}_{\ge 0}$ be a statistic on a set $S$, if a group $G$ acts on $S$ in such a way that every orbit $\mathcal{O}$ of the action has an average value of $\mathrm{st}$ equal to $\frac{\mathrm{st}~\mathcal{O}}{\# \mathcal{O}} = d$, for some constant $d$, where $\mathrm{st}~\mathcal{O}$ is the sum of the statistic values for all elements in $\mathcal{O}$ and $\#~\mathcal{O}$ is the number of elements in $\mathcal{O}$, then we say that $\mathrm{st}$ is \textbf{$d$-mesic}. A statistic is \textbf{homomesic} if it is $d$-mesic for some $d$. Note that any linear combination of homomesic statistics will also be homomesic. Consequently, the collection of homomesic statistics on $S$ with respect to the action of $G$ forms a vector space. 

Many statistics of interest in the study of homomesy are defined as linear combinations of the following functions on order ideals. For a poset $\mathcal{P}$, define the \textbf{order ideal indicator function} $\mathbbm{1}_p:\mathcal{J}(\mathcal{P})\to\mathbb{R}$ for $p\in \mathcal{P}$ by $$\mathbbm{1}_p(I)=\begin{cases}
        1, & if~p\in I \\
        0, & otherwise,
    \end{cases}$$
the \textbf{antichain indicator function} $\mathcal{T}_p^-:\mathcal{J}(\mathcal{P})\to\mathbb{R}$ for $p\in \mathcal{P}$ by $$\mathcal{T}_p^-(I)=\begin{cases}
        1, & \text{if}~p\in \mathrm{Max}(I) \\
        0, & otherwise,
    \end{cases}$$ and $\mathcal{T}_p^+:\mathcal{J}(\mathcal{P})\to\mathbb{R}$ by $$\mathcal{T}_p^+(I)=\begin{cases}
        1, & \text{if}~p\notin I~\text{and}~I\cup\{p\}\in \mathcal{J}(\mathcal{P}) \\
        0, & otherwise.
    \end{cases}$$
Moreover, the \textbf{toggleability statistic}, $\mathcal{T}_p: \mathcal{J}(\mathcal{P})\rightarrow \mathcal{J(\mathcal{P})}$ is defined by $$\mathcal{T}_p(I)=\mathcal{T}_p^+(I)-\mathcal{T}_p^-(I)=\begin{cases}
        1, & if~p\in \mathrm{Min}(\mathcal{P}\backslash I) \\
        -1, & if~p\in \mathrm{Max}(I) \\
        0, & otherwise.
    \end{cases}$$
Toggleability statistics are of special interest in the study of homomesy under rowmotion due to the following result by Striker, which implies that any linear combination of toggleability statistics and a constant is a homomesic statistic under rowmotion. 

\begin{lemma}\cite{striker2015toggle}
    For any $p \in P$, the statistic $\mathcal{T}_p$ is $0$-mesic under rowmotion.
\end{lemma}

We are now in a position to introduce the aforementioned vector spaces of interest in this paper as well as state our motivating main result. For a finite poset $\mathcal{P}$, let $I_H(\mathcal{P})$ and $A_H(\mathcal{P})$ denote the subspaces of homomesic statistics within $\text{Span}_{\mathbb{R}}(\mathbbm{1}_p~|~p\in\mathcal{P})$ and $\text{Span}_{\mathbb{R}}(\mathcal{T}^-_p~|~p\in\mathcal{P})$, respectively. The \textbf{toggleability spaces} of $\mathcal{P}$, denoted $I_T(\mathcal{P})$ and $A_T(\mathcal{P})$, are the collections of elements in $I_H(\mathcal{P})$ and $A_H(\mathcal{P})$, respectively, which are linear combinations of toggleability statistics $\mathcal{T}_p$ and a constant. We refer to $I_T(\mathcal{P})$ and $A_T(\mathcal{P})$ as the \textbf{order ideal} and \textbf{antichain toggleability spaces} of $\mathcal{P}$, respectively. Toggleability spaces were first considered in \cite{defant2021homomesy}. They are studied again in \cite{mertin2024toggleability}, where the authors completely describe the order ideal and antichain toggleability spaces for the family of fence posets. Here, as noted in Section~\ref{sec:intro}, our main goal is to establish conjectured dimensions for the toggleability spaces associated with the four families of posets introduced above. In particular, our aim is to prove Theorem~\ref{thm:main} from Section~\ref{sec:intro}, coming from the combinatorial conjectures detailed in Table 6.1 of~\cite{defant2021homomesy}, which we state again below for convenience. 

\begin{thm}\label{thm:main}~
    \begin{enumerate}
    \item[$(a)$] If $\mathcal{P}=[m]\times [n]$ for $m,n\ge 2$, then $\dim I_T(\mathcal{P})=\dim A_T(\mathcal{P})=\mathrm{rk}(\mathcal{P})+1=n+m-1$.
    \item[$(b)$] If $\mathcal{P}=([n]\times[n])\backslash \mathfrak{S}_2$ for $n\ge 2$, then $\dim I_T(\mathcal{P})=\dim A_T(\mathcal{P})=\mathrm{rk}(\mathcal{P})+1=2n-1$.
    \item[$(c)$] If $\mathcal{P}=\Phi^+(A_n)$ for $n\ge 2$, then $\dim I_T(\mathcal{P})=\dim A_T(\mathcal{P})=\mathrm{rk}(\mathcal{P})+1=n$.
    \item[$(d)$] If $\mathcal{P}=\Phi^+(B_n)$ for $n\ge 2$, then $\dim I_T(\mathcal{P})=\dim A_T(\mathcal{P})=\mathrm{rk}(\mathcal{P})+1=2n-1$.
\end{enumerate}
\end{thm}

\noindent
The proof of Theorem~\ref{thm:main} is broken up into two sections, with Section~\ref{sec:conj1} focusing on $I_T(\mathcal{P})$ and Section~\ref{sec:antichain} focusing on $A_T(\mathcal{P})$.

\section{Order Ideals}\label{sec:conj1}

In this section, we prove the conjectures of Hopkins, Poznanovi\'{c}, and Propp \cite{defant2021homomesy} concerning $\dim I_T(\mathcal{P})$ for posets $\mathcal{P}$ belonging to the four families of posets introduced in Section~\ref{sec:background}. We collect these results into Theorem \ref{thm:main1} below.

\begin{thm}\label{thm:main1}~
\begin{enumerate}
    \item[$(a)$] If $\mathcal{P}=[m]\times [n]$ for $m,n\ge 2$, then $\dim I_T(\mathcal{P})=\mathrm{rk}(\mathcal{P})+1=n+m-1$.
    \item[$(b)$] If $\mathcal{P}=([n]\times[n])\backslash \mathfrak{S}_2$ for $n\ge 2$, then $\dim I_T(\mathcal{P})=\mathrm{rk}(\mathcal{P})+1=2n-1$.
    \item[$(c)$] If $\mathcal{P}=\Phi^+(A_n)$ for $n\ge 2$, then $\dim I_T(\mathcal{P})=\mathrm{rk}(\mathcal{P})+1=n$.
    \item[$(d)$] If $\mathcal{P}=\Phi^+(B_n)$ for $n\ge 2$, then $\dim I_T(\mathcal{P})=\mathrm{rk}(\mathcal{P})+1=2n-1$.
\end{enumerate}
\end{thm}

\begin{rem}
    It is well known that the rank values of the posets considered in Theorem~\ref{thm:main1} are equal to the implied numerical values, e.g., $\mathrm{rk}([m]\times[n])=n+m-2$. Consequently, the corresponding equalities are assumed ongoing.
\end{rem}

We break the proof of Theorem~\ref{thm:main1} into two sections, first showing $\dim I_T(\mathcal{P})\le \mathrm{rk}(\mathcal{P})+1$ in Section~\ref{sec:UB} and then $\dim I_T(\mathcal{P})\ge \mathrm{rk}(\mathcal{P})+1$ in Section~\ref{sec:LB}. Following the proof of Theorem~\ref{thm:main1}, we apply the results of Sections~\ref{sec:UB} and~\ref{sec:LB} to an additional family of posets defined by integer partitions in Section~\ref{sec:ip}.

\subsection{Upper bound: $\dim I_T(\mathcal{P})\le \mathrm{rk}(\mathcal{P})+1$}\label{sec:UB}

In this subsection, we show that $\dim I_T(\mathcal{P})\le \mathrm{rk}(\mathcal{P})+1$ for the posets $\mathcal{P}$ of Theorem~\ref{thm:main1}. To do so, we require Theorems~\ref{thm:Diamond} and~\ref{thm:diamond} as well as Lemmas~\ref{lem:pchain} and \ref{lem:root_zero} below, which identify restrictions satisfied by elements of $I_T(\mathcal{P})$ for certain posets $\mathcal{P}$. 

We start with Theorem~\ref{thm:Diamond}, which concerns restrictions involving pairs of elements $p_1$ and $p_4$ in a poset $\mathcal{P}$ for which there exists $p_2,p_3\in\mathcal{P}$ satisfying $p_1\precdot p_2\precdot p_4$ and $p_1\precdot p_3\precdot p_4$. Then, in Theorem~\ref{thm:diamond}, we show that Theorem~\ref{thm:Diamond} holds for all appropriate elements of posets defined by connected diagrams which are both row and column convex.

\begin{thm}\label{thm:Diamond}
    Let $\mathcal{P}$ be a poset with $p_1,p_2,p_3,p_4\in\mathcal{P}$ satisfying $p_1\precdot p_2\precdot p_4$ and $p_1\precdot p_3\precdot p_4$. 
    Setting $S_i=\textrm{Min}(\mathcal{P}\backslash I(p_i))$ for $i\in[3]$, $S_{2,3}=\textrm{Min}(\mathcal{P}\backslash I(p_2,p_3))$, $S_2^1=S_2\backslash S_1$, and $S_3^1=S_3\backslash S_1$,
    suppose that
    \begin{itemize}
        \item[$(a)$] $S_2^1, S_3^1, \{p_4\} \subseteq S_{2,3}$,
    \item[$(b)$] $S_1, S_{2,3} \setminus \{p_4\} \subseteq S_2 \cup S_3$, and
    \item[($c$)] $S_{2,3} \cap S_1 = S_2 \cap S_3$.
    \end{itemize}
    Then for $$f=c+\sum_{p\in\mathcal{P}}c_p\mathcal{T}_p\in I_T(\mathcal{P}),$$ we have $c_{p_1}=c_{p_4}$.
\end{thm}

In Example~\ref{ex:Diamond}, we provide a sample application of Theorem~\ref{thm:Diamond}.

\begin{exmp}\label{ex:Diamond}
    Let $\mathcal{P}$ be the poset defined by the diagram $D$ illustrated in Figure~\ref{fig:Diamond}. For the elements $p_1,p_2,p_3,p_4\in\mathcal{P}$ labeled in $D$, we have $p_1\precdot p_2\precdot p_4$ and $p_1\precdot p_3\precdot p_4$. Moreover, in the notation of Theorem~\ref{thm:Diamond}, $$S_1=\{s_1,s_3,p_3\},\quad S_2=\{s_1,p_3\},\quad S_3=\{s_1,s_3,s_6\},\quad S_{2,3}=\{s_1,s_6,p_4\},\quad S_2^1=\emptyset,\quad S_3^1=\{s_6\},$$ $$S_{2,3}\backslash(S_2^1\cup S_3^1\cup\{p_4\})=\{s_1\},\quad (S_2\cap S_1)\backslash S_{2,3}=\{p_3\},\quad\text{and}\quad (S_3\cap S_1)\backslash S_{2,3}=\{s_3\}.$$ One can check that $p_1,p_2,p_3,p_4\in \mathcal{P}$ satisfy the hypotheses of Theorem~\ref{thm:Diamond}. Thus, applying Theorem~\ref{thm:Diamond}, if $\displaystyle f=c+\sum_{p\in\mathcal{P}}c_p\mathcal{T}_p\in I_T(\mathcal{P})$, then $c_{p_1}=c_{p_4}$.

    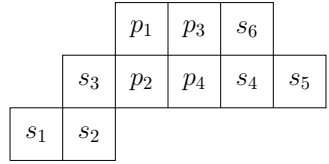
\begin{figure}[H]
        \centering
        $$\scalebox{0.7}{\begin{tikzpicture}
            \draw (0,0)--(2,0)--(2,1)--(6,1)--(6,2)--(5,2)--(5,3)--(2,3)--(2,2)--(1,2)--(1,1)--(0,1)--(0,0);
            \draw (1,0)--(1,1)--(2,1)--(2,2)--(5,2);
            \draw (3,1)--(3,3);
            \draw (4,1)--(4,3);
            \draw (5,1)--(5,2);
            \node at (0.5,0.5) {\Large $s_1$};
            \node at (1.5,0.5) {\Large $s_2$};
            \node at (1.5,1.5) {\Large $s_3$};
            \node at (2.5,1.5) {\Large $p_2$};
            \node at (3.5,1.5) {\Large $p_4$};
            \node at (4.5,1.5) {\Large $s_4$};
            \node at (5.5,1.5) {\Large $s_5$};
            \node at (2.5,2.5) {\Large $p_1$};
            \node at (3.5,2.5) {\Large $p_3$};
            \node at (4.5,2.5) {\Large $s_6$};
        \end{tikzpicture}}$$
        \caption{Example elements described in Theorem~\ref{thm:Diamond}}
        \label{fig:Diamond}
    \end{figure}

\end{exmp}

\begin{proof}[Proof of Theorem \ref{thm:Diamond}]
   First, we show that conditions $(a)$, $(b)$, and $(c)$ in the statement of Theorem~\ref{thm:Diamond} together imply the following:
    \begin{itemize}
        \item[$(a')$] $S_2^1\cup S_3^1\cup [S_{2,3}\backslash(S^1_2\cup S^1_3\cup\{p_4\})]\cup \{p_4\}=S_{2,3}$
        \item[$(b')$] $[(S_2\cap S_1)\backslash S_{2,3}]\cup[(S_3\cap S_1)\backslash S_{2,3}]\cup [S_{2,3}\backslash (S_2^1\cup S_3^1\cup\{p_4\})]=S_1$,
        \item[$(c')$] $S_i^1\cup [S_{2,3}\backslash(S^1_2\cup S^1_3\cup\{p_4\})]\cup [(S_i\cap S_1)\backslash S_{2,3}]=S_i$ for $i=2,3$, and
        \item[$(d')$]$S_2^1, S_3^1, S_{2,3}\setminus(S_2^1 \cup S_3^1 \cup \{p_4\})$, $(S_2 \cap S_1) \setminus S_{2,3}$, $(S_3 \cap S_1)\setminus S_{2,3}$, and $\{p_4\}$ are mutually disjoint.
\end{itemize}
The fact that $(a)$ implies $(a')$ is immediate. Moving to $(b')$, observe by $(b)$ that $S_{2,3} \setminus (S_2^1 \cup S_3^1 \cup \{p_4\}) \subseteq S_1$. Therefore, $[(S_2\cap S_1)\backslash S_{2,3}]\cup[(S_3\cap S_1)\backslash S_{2,3}]\cup [S_{2,3}\backslash (S_2^1\cup S_3^1\cup\{p_4\})] \subseteq S_1$. Now, consider $x \in S_1$. Note that $p_4 \not\in S_1$, i.e., $x\neq p_4$, because $p_2 \in \mathcal{P}\setminus I(p_1)$ and $p_2 \precdot p_4$ so $p_4$ is not minimal here. 
Thus, if $x \in S_{2,3}$, then $x \in [S_{2,3}\backslash (S_2^1\cup S_3^1\cup\{p_4\})]$. On the other hand, if $x \not\in S_{2,3}$, then $x \in [(S_2 \cup S_3) \cap S_1] \setminus S_{2,3}$ by $(b)$.  Therefore, $[(S_2\cap S_1)\backslash S_{2,3}]\cup[(S_3\cap S_1)\backslash S_{2,3}]\cup [S_{2,3}\backslash (S_2^1\cup S_3^1\cup\{p_4\})] \supseteq S_1$, and we have now shown $(b')$.

For $(c')$, as stated previously, we have that $S_{2,3} \setminus (S_2^1 \cup S_3^1 \cup \{p_4\}) \subseteq S_1$ and so, considering $(c)$, we have $S_{2,3} \setminus (S_2^1 \cup S_3^1 \cup \{p_4\}) \subseteq S_2 \cap S_3$. Thus, for $i=2,3$, we have that $S_{2,3} \setminus (S_2^1 \cup S_3^1 \cup \{p_4\}) \subseteq S_i$ and, consequently, $S_i^1\cup [S_{2,3}\backslash(S^1_2\cup S^1_3\cup\{p_4\})]\cup [(S_i\cap S_1)\backslash S_{2,3}] \subseteq S_i$.  Consider $x \in S_i$ for $i=2$ or $3$. If $x \not\in S_1$, then $x\in S_i^1$; if $x\in S_1$ and $x \not\in S_{2,3}$, then $x \in [(S_i \cap S_1)\setminus S_{2,3}]$; and, finally,  if $x \in S_1$ and $x \in S_{2,3}$, then $x \in S_{2,3} \setminus (S_2^1 \cup S_3^1 \cup \{p_4\})$ since $S_{2,3} \setminus (S_2^1 \cup S_3^1 \cup \{p_4\}) \subseteq S_1$ and this set subtraction does not remove any elements of $S_{2,3}$ also contained $S_1$. Thus, $S_i^1\cup [S_{2,3}\backslash(S^1_2\cup S^1_3\cup\{p_4\})]\cup [(S_i\cap S_1)\backslash S_{2,3}] \supseteq S_1$, and we have shown $(c')$. Finally, for $(d')$, observe that $p_3 \in \mathcal{P} \setminus I(p_2)$ and $p_3 \precdot p_4$ so $p_4 \not\in S_2$. Similar reasoning yields $p_4 \not\in S_3$. From here, each pair can be checked to be disjoint either by definition or condition $(c)$.

Now, assume that $f=\sum_{p\in\mathcal{P}}a_p\mathbbm{1}_p.$ Then letting $$C=f(I(p_2,p_3))-f(I(p_2))-f(I(p_3))+f(I(p_1)),$$ we have that
    \begin{align*}
        C&=\left(\sum_{p\preceq p_1}a_p+\sum_{\substack{p\preceq p_2 \\ p\npreceq p_1}}a_p+\sum_{\substack{p\preceq p_3 \\ p\npreceq p_1}}a_p\right)-\left(\sum_{p\preceq p_1}a_p+\sum_{\substack{p\preceq p_2 \\ p\npreceq p_1}}a_p+\right)-\left(\sum_{p\preceq p_1}a_p+\sum_{\substack{p\preceq p_3 \\ p\npreceq p_1}}a_p\right)+\sum_{p\preceq p_1}a_p\\
        &=0.
    \end{align*}
    On the other hand, expressing $C$ in terms of the $c_p$ for $p\in\mathcal{P}$, we have
    \begin{align*}
        C&=\left(c+c_{p_4}-c_{p_2}-c_{p_3}+\sum_{p\in S_{2,3}\backslash\{p_4\}}c_p\right)-\left(c-c_{p_2}+\sum_{p\in S_2}c_p\right)-\left(c-c_{p_3}+\sum_{p\in S_3}c_p\right)\\
        &~~~~~~+\left(c-c_{p_1}+\sum_{p\in S_1}c_p\right)\\
        &=\left(c+c_{p_4}-c_{p_2}-c_{p_3}+\sum_{p\in S^1_2}c_p+\sum_{p\in S^1_3}c_p+\sum_{p\in S_{2,3}\backslash (S_2^1\cup S_3^1\cup\{p_4\})}c_p\right)\\
        &~~~~~~-\left(c-c_{p_2}+\sum_{p\in S^1_2}c_p+\sum_{p\in S_{2,3}\backslash (S_2^1\cup S_3^1\cup\{p_4\})}c_p+\sum_{p\in (S_2\cap S_1)\backslash S_{2,3}}c_p\right)\\
        &~~~~~~-\left(c-c_{p_3}+\sum_{p\in S^1_3}c_p+\sum_{p\in S_{2,3}\backslash (S_2^1\cup S_3^1\cup\{p_4\})}c_p+\sum_{p\in (S_3\cap S_1)\backslash S_{2,3}}c_p\right)\\
        &~~~~~~+\left(c-c_{p_1}+\sum_{p\in S_{2,3}\backslash (S_2^1\cup S_3^1\cup\{p_4\})}c_p+\sum_{p\in (S_2\cap S_1)\backslash S_{2,3}}c_p+\sum_{p\in (S_3\cap S_1)\backslash S_{2,3}}c_p\right)\\
        &=c_{p_4}-c_{p_1}.
    \end{align*}
    Consequently, $c_{p_4}-c_{p_1}=C=0$; that is, $c_{p_4}=c_{p_1}$, as desired.
\end{proof}

\begin{rem}
    One can show that conditions $(a)$, $(b)$, and $(c)$ of Theorem~\ref{thm:Diamond} are, in fact, equivalent to conditions $(a')$, $(b')$, $(c')$, and $(d')$ from the proof above.
\end{rem}

\begin{thm}\label{thm:diamond}
    Let $\mathcal{P}$ be a poset defined by a connected diagram which is both row and column convex. Suppose that $p_1,p_2,p_3,p_4\in\mathcal{P}$ satisfy $p_1\precdot p_2\precdot p_4$ and $p_1\precdot p_3\precdot p_4$. Then for $$f=c+\sum_{p\in\mathcal{P}}c_p\mathcal{T}_p\in I_T(\mathcal{P}),$$ we have $c_{p_1}=c_{p_4}$.
\end{thm}
\begin{proof}
    Identify elements of $\mathcal{P}$ with the row-column coordinates of their cells in the diagram representation. Recall that rows are numbered from top to bottom and columns from left to right. Let $M$ denote the collection of minimal elements of $\mathcal{P}$. Then for $p=(i,j)\in\mathcal{P}$, set $$M^1_1=\{(i^*,j+1)\in\mathcal{P}\backslash M~|~i^*=\min\{k~|~(k,j+1)\in\mathcal{P},~k\le i\}\},$$ $$M^2_1=\{(i^*,j+1)\in M~|~i^*=\min\{k~|~(k,j+1)\in\mathcal{P},~k\le i\}\},$$ $$M_2^1=\{(i+1,j^*)\in\mathcal{P}\backslash M~|~j^*=\min\{k~|~(i+1,k)\in\mathcal{P},~k\le j\}\},$$ $$M_2^2=\{(i+1,j^*)\in M~|~j^*=\min\{k~|~(i+1,k)\in\mathcal{P},~k\le j\}\},$$ and $$M_3=\{(k,\ell)\in M~|~k>i+1\}\cup\{(k,\ell)\in M~|~\ell>j+1\}.$$ Note that for $i=1,2$, exactly one of the following holds:
    \begin{itemize}
        \item $|M_i^1|=|M_i^2|=0$,
        \item $|M_i^1|=1$ and $|M_i^2|=0$, or 
        \item $|M_i^1|=0$ and $|M_i^2|=1$.
    \end{itemize} 
    We claim that $$\mathrm{Min}(\mathcal{P}\backslash I(p))=M_1^1\cup M_1^2\cup M_2^1\cup M_1^2\cup M_3.$$ Note that all elements of $\mathrm{Min}(\mathcal{P}\backslash I(p))$ must be of the form $(k,\ell)$ for $k>i$ or $\ell>j$. Now, the fact that $$\mathrm{Min}(\mathcal{P}\backslash I(p))\supseteq M_1^1\cup M_1^2\cup M_2^1\cup M_1^2$$ should be clear. Thus, to establish the claim, we show that $$\mathrm{Min}(\mathcal{P}\backslash I(p))\backslash 
    (M_1^1\cup M_1^2\cup M_2^1\cup M_1^2)=M_3.$$ To do so, it suffices to show that if $(k,\ell)\in \mathrm{Min}(\mathcal{P}\backslash I(p))$ for $k>i+1$ or $\ell>j+1$, then $(k,\ell)\in M$. So, take $(k,\ell)\in \mathrm{Min}(\mathcal{P}\backslash I(p))$ for $k>i+1$ and assume for a contradiction that $(k,\ell)\notin M$; a similar argument applies if one instead assumes that $\ell>j+1$. Since $(k,\ell)\in \mathrm{Min}(\mathcal{P}\backslash I(p))$, it follows that $(k-1,r)\notin \mathcal{P}$ for $r\le \ell$. On the other hand, since $(k,\ell)\notin M$, there exists $(s,t)\in I(p)\subseteq\mathcal{P}$ with $s<k-1$ and $t\le \ell$. Since the diagram defining $\mathcal{P}$ is assumed to be connected, it follows that there must exist $(s',\ell)\in\mathcal{P}$ with $s'<k-1$; but since $(k-1,\ell)\notin\mathcal{P}$, this contradicts our assumption that the diagram defining $\mathcal{P}$ is row convex. Therefore, $\mathrm{Min}(\mathcal{P}\backslash I(p))\backslash (M_1^1\cup M_1^2\cup M_2^1\cup M_2^2)=M_3$ so that $$\mathrm{Min}(\mathcal{P}\backslash I(p))=M_1^1\cup M_1^2\cup M_2^1\cup M_1^2\cup M_3,$$ as claimed.

    Now, for $p_1,p_2,p_3,p_4\in\mathcal{P}$ satisfying $p_1\precdot p_2\precdot p_4$ and $p_1\precdot p_3\precdot p_4$, note that there exist $i,j\in\mathbb{Z}$ such that $p_1=(i,j)$, $p_2=(i+1,j)$, $p_3=(i,j+1)$, and $p_4=(i+1,j+1)$. Letting $S_{2,3}=\mathrm{Min}(\mathcal{P}\backslash I(p_2,p_3))$ and $S_{i}=\mathrm{Min}(\mathcal{P}\backslash I(p_i))$ for $i\in[3]$, using the claim established above we find that 
    \begin{align*}
        S_1=\{(i^*,j+1)\in&\mathcal{P}\backslash M~|~i^*=\min\{k~|~(k,j+1)\in\mathcal{P},~k\le i\}\}\cup \{(k,\ell)\in M~|~k>i+1\}\\
        &\cup\{(k,\ell)\in M~|~\ell>j+1\}\cup\{(i^*,j+1)\in M~|~i^*=\min\{k~|~(k,j+1)\in\mathcal{P},~k\le i\}\}\\ 
        &\cup\{(i+1,j^*)\in\mathcal{P}\backslash M~|~j^*=\min\{k~|~(i+1,k)\in\mathcal{P},~k\le j\}\}\\
        &\cup\{(i+1,j^*)\in M~|~j^*=\min\{k~|~(i+1,k)\in\mathcal{P},~k\le j\}\},
    \end{align*}
    \begin{align*}
        S_2=\{(i^*,j+1)\in&\mathcal{P}\backslash M~|~i^*=\min\{k~|~(k,j+1)\in\mathcal{P},~k\le i+1\}\}\cup \{(k,\ell)\in M~|~k>i+2\}\\
        &\cup\{(k,\ell)\in M~|~\ell>j+1\}\cup\{(i^*,j+1)\in M~|~i^*=\min\{k~|~(k,j+1)\in\mathcal{P},~k\le i+1\}\}\\ 
        &\cup\{(i+2,j^*)\in\mathcal{P}\backslash M~|~j^*=\min\{k~|~(i+2,k)\in\mathcal{P},~k\le j\}\}\\
        &\cup\{(i+2,j^*)\in M~|~j^*=\min\{k~|~(i+2,k)\in\mathcal{P},~k\le j\}\},
    \end{align*}
   \begin{align*}
        S_3=\{(i^*,j+2)\in&\mathcal{P}\backslash M~|~i^*=\min\{k~|~(k,j+2)\in\mathcal{P},~k\le i\}\}\cup \{(k,\ell)\in M~|~k>i+1\}\\
        &\cup\{(k,\ell)\in M~|~\ell>j+2\}\cup\{(i^*,j+2)\in M~|~i^*=\min\{k~|~(k,j+2)\in\mathcal{P},~k\le i\}\}\\ 
        &\cup\{(i+1,j^*)\in\mathcal{P}\backslash M~|~j^*=\min\{k~|~(i+1,k)\in\mathcal{P},~k\le j+1\}\}\\
        &\cup\{(i+1,j^*)\in M~|~j^*=\min\{k~|~(i+1,k)\in\mathcal{P},~k\le j+1\}\},
    \end{align*}
    and
    \begin{align*}
        S_{2,3}=\{(i^*,j+2)\in&\mathcal{P}\backslash M~|~i^*=\min\{k~|~(k,j+2)\in\mathcal{P},~k\le i\}\}\cup \{(k,\ell)\in M~|~k>i+2\}\\
        &\cup\{(k,\ell)\in M~|~\ell>j+2\}\cup\{(i^*,j+2)\in M~|~i^*=\min\{k~|~(k,j+2)\in\mathcal{P},~k\le i\}\}\\ 
        &\cup\{(i+2,j^*)\in\mathcal{P}\backslash M~|~j^*=\min\{k~|~(i+2,k)\in\mathcal{P},~k\le j\}\}\\
        &\cup\{(i+2,j^*)\in M~|~j^*=\min\{k~|~(i+2,k)\in\mathcal{P},~k\le j\}\}\cup\{p_4=(i+1,j+1)\}.
    \end{align*}
    Consequently,
    \begin{align*}
        S_2^1=S_2\backslash S_1&=\{(i+2,j^*)\in\mathcal{P}\backslash M~|~j^*=\min\{k~|~(i+2,k)\in\mathcal{P},~k\le j\}\},
    \end{align*}
    \begin{align*}
        S_3^1=S_3\backslash S_1&=\{(i^*,j+2)\in\mathcal{P}\backslash M~|~j^*=\min\{k~|~(k,j+2)\in\mathcal{P},~k\le i\}\},
    \end{align*}
    and
    \begin{align*}
        S_2\cap S_3=S_1\cap S_{2,3}=\{(k,\ell)\in M~|~k>i+1\}\cup\{(k,\ell)\in M~|~\ell>j+1\}.
    \end{align*}
    It is straightforward to verify that the collections above satisfy the hypotheses of Theorem~\ref{thm:Diamond}. Consequently, applying Theorem~\ref{thm:Diamond}, $c_{p_4}=c_{p_1}$, as desired.
\end{proof}

Along with Theorem~\ref{thm:diamond}, we also require the following additional lemmas concerning further restrictions on elements of $I_T(\mathcal{P})$.

\begin{lemma}\label{lem:pchain}
    Let $\mathcal{P}$ be a poset, $M=\mathrm{Min}(\mathcal{P})$, and $$f=c+\sum_{p\in\mathcal{P}}c_p\mathcal{T}_p\in I_T(\mathcal{P}).$$ Then $c=-\sum_{p\in M}c_{p}$.
\end{lemma}
\begin{proof}
    Since $f=c+\sum_{p\in\mathcal{P}}c_p\mathcal{T}_p$, we find that $f(\emptyset)=c+\sum_{p\in M}c_{p}$. On the other hand, 
    since $f\in \text{Span}_{\mathbb{R}}(\mathbbm{1}_p~|~p\in\mathcal{P})$, it follows that $f(\emptyset)=0$. Consequently, $c+\sum_{p\in M}c_{p}=f(\emptyset)=0$; that is, $c=-\sum_{p\in M}c_{p}$, as desired.
\end{proof}

\begin{lemma}\label{lem:root_zero} 
Let $\mathcal{P}$ be a poset defined by a connected diagram with $p,q_1,q_2\in\mathcal{P}$ where $$\{r\in\mathcal{P}~|~r\preceq q_1\}\cap \{r\in\mathcal{P}~|~r\preceq q_2\}=\emptyset$$ and $p$ is the unique element in $\mathcal{P}$ satisfying $q_1,q_2\precdot p$. Then for $$f = c + \sum_{p\in P}c_{p}\mathcal{T}_{p} \in I_T(P),$$ we have $c_{p} = 0$.
\end{lemma}

In Example~\ref{ex:root_zero}, we provide a sample application of Theorem~\ref{lem:root_zero}.

\begin{exmp}\label{ex:root_zero}
Let $\mathcal{P}$ be the poset defined for the diagram $D$ illustrated in Figure~\ref{fig:root_zero}. The elements $p,q_1,q_2\in\mathcal{P}$ labeled in $D$ satisfy the hypotheses of Lemma~\ref{lem:root_zero}. Thus, applying Lemma~\ref{lem:root_zero}, if $f=c+\sum_{p\in\mathcal{P}}c_p\mathcal{T}_p\in I_T(\mathcal{P})$, then $c_{p}=0$.

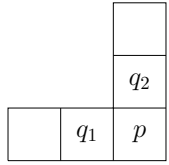
\begin{figure}[H]
    \centering
    $$\scalebox{0.7}{\begin{tikzpicture}
        \draw (0,0)--(3,0)--(3,3)--(2,3)--(2,1)--(0,1)--(0,0);
        \draw (1,0)--(1,1);
        \draw (2,0)--(2,1)--(3,1);
        \draw (2,2)--(3,2);
        \node at (2.5,0.5) {\Large $p$};
        \node at (1.5,0.5) {\Large $q_1$};
        \node at (2.5,1.5) {\Large $q_2$};
    \end{tikzpicture}}$$
    \caption{Example elements described in Lemma~\ref{lem:root_zero}}
    \label{fig:root_zero}
\end{figure}
\end{exmp}

\begin{proof}[Proof of Lemma \ref{lem:root_zero}]
Assume that $f=\sum_{p\in\mathcal{P}}a_p\mathbbm{1}_p$. Then, letting $$C=f(I(q_1, q_2)) - f(I(q_1))) - f(I(q_2)) + f(\emptyset),$$ we have that $$C=\left(\sum_{q_1\succeq r\in\mathcal{P}}a_r+\sum_{q_2\succeq r\in\mathcal{P}}a_r\right)-\sum_{q_1\succeq r\in\mathcal{P}}a_r-\sum_{q_2\succeq r\in\mathcal{P}}a_r+0=0.$$
Now, let $M=\mathrm{Min}(\mathcal{P})$, $L_1=\mathrm{Min}(\mathcal{P}\backslash I(q_1))\backslash M$, and $L_2=\mathrm{Min}(\mathcal{P}\backslash I(q_2))\backslash M$. By assumption, we have that $L_1\cap L_2=\emptyset$. Thus,
\begin{align*}
    C&=\left(c_p - c_{q_1} - c_{q_2} + \sum_{\substack{m\in M \\ m\nprec q_1,q_2}}c_m+\sum_{r\in L_1}c_r+\sum_{r\in L_2}c_r\right) - \left(-c_{q_1} + \sum_{\substack{m\in M \\ m\npreceq q_1,q_2}}c_m+\sum_{q_2\succeq m\in M}c_m+\sum_{r\in L_1}c_r\right)\\
    &~~~~~~~~- \left(-c_{q_2} + \sum_{\substack{m\in M \\ m\npreceq q_1,q_2}}c_m+\sum_{q_1\succeq m\in M}c_m+\sum_{r\in L_2}c_r\right)+ \left(\sum_{\substack{m\in M \\ m\npreceq q_1,q_2}}c_m+\sum_{q_1\succeq m\in M}c_m+\sum_{q_2\succeq m\in M}c_m\right)\\
    & = c_p. 
\end{align*}
Consequently, $c_p=C=0$; that is, $c_p=0$, as desired.
\end{proof}

Applying the results above, we obtain upper bounds on $\dim I_T(\mathcal{P})$ for our four primary types of posets.

\begin{proposition}\label{prop:UB}~
    \begin{enumerate}
    \item[$(a)$] If $\mathcal{P}=[m]\times [n]$ for $m,n\ge 2$, then $\dim I_T(\mathcal{P})\le \mathrm{rk}(\mathcal{P})+1= n+m-1$.
    \item[$(b)$] If $\mathcal{P}=([n]\times[n])\backslash \mathfrak{S}_2$ for $n\ge 2$, then $\dim I_T(\mathcal{P})\le \mathrm{rk}(\mathcal{P})+1= 2n-1$.
    \item[$(c)$] If $\mathcal{P}=\Phi^+(A_n)$ for $n\ge 2$, then $\dim I_T(\mathcal{P})\le \mathrm{rk}(\mathcal{P})+1= n$.
    \item[$(d)$] If $\mathcal{P}=\Phi^+(B_n)$ for $n\ge 2$, then $\dim I_T(\mathcal{P})\le \mathrm{rk}(\mathcal{P})+1= 2n-1$.
\end{enumerate}
\end{proposition}
\begin{proof} In what follows, for each choice of $\mathcal{P}$, we let $$f=c+\sum_{(i,j)\in\mathcal{P}}c_{i,j}\mathcal{T}_{(i,j)}\in I_T(\mathcal{P}).$$
    \begin{enumerate}
        \item[$(a)$] Take $\mathcal{P}=[m]\times [n]$ for $m,n\ge 2$. Note that, applying Lemma~\ref{lem:pchain}, it follows that $c=-c_{1,1}$. In addition, we claim that $c_{i,j}=c_{k,\ell}$ for $(i,j),(k,\ell)\in\mathcal{P}$ satisfying $i-j=k-\ell$. To establish the claim, it suffices to show that if $(i,j),(i+1,j+1)\in\mathcal{P}$, then $c_{i,j}=c_{i+1,j+1}$. To this end, note that if $(i,j),(i+1,j+1)\in\mathcal{P}$, then $(i+1,j),(i,j+1)\in\mathcal{P}$ with $(i,j)\precdot (i+1,j)\precdot (i+1,j+1)$ and $(i,j)\precdot (i,j+1)\precdot (i+1,j+1)$. Thus, since $\mathcal{P}$ can be defined by a connected diagram which is both row and column convex (being an $n\times m$ rectangle), applying Theorem~\ref{thm:diamond}, we find that $c_{i,j}=c_{i+1,j+1}$ and the claim follows. Thus, we have shown that 
        \begin{itemize}
            \item $c=-c_{1,1}$ and
            \item $c_{i,j}=c_{k,\ell}$ for $(i,j),(k,\ell)\in\mathcal{P}$ satisfying $i-j=k-\ell$.
        \end{itemize}
        Consequently, it follows that $\dim I_T(\mathcal{P})$ is at most $$|\{k~|~i-j=k~\text{and}~(i,j)\in[m]\times[n]\}|=n+m-1=\mathrm{rk}([m]\times[n])+1,$$ as desired.

        \item[$(b)$] Take $\mathcal{P}=([n]\times[n])\backslash \mathfrak{S}_2$ for $n\ge 2$. Arguing exactly as in the proof of $\mathcal{P}=[m]\times [n]$ for $m,n\ge 2$, we find that 
        \begin{itemize}
            \item $c=-c_{1,1}$ and
            \item $c_{i,j}=c_{k,\ell}$ for $(i,j),(k,\ell)\in\mathcal{P}$ satisfying $i-j=k-\ell$.
        \end{itemize}
        Consequently, it follows that $\dim I_T(([n]\times [n])\backslash \mathfrak{S}_2)$ is at most $$|\{k\in\mathbb{Z}~|~k\in (0,n-1]\}|+n=2n-1=\mathrm{rk}(\mathcal{P})+1,$$ as desired.

        \item[$(c)$] Take $\mathcal{P}=\Phi^+(A_n)$ for $n\ge 2$. Applying Lemma~\ref{lem:pchain}, we find that $c = - \sum_{i=0}^{n-1} c_{1+i, n-i}$. Next, note that if $(i,j)\in\mathcal{P}$ satisfies $i+j=n+2$, then 
    \begin{itemize}
        \item $(i-1,j),(i,j-1)\in\mathcal{P}$ are minimal elements of $\mathcal{P}$ and
        \item $(i,j)$ is the unique element of $\mathcal{P}$ satisfying $(i-1,j),(i,j-1)\precdot (i,j)$. 
    \end{itemize}
    Thus, applying Lemma~\ref{lem:root_zero}, it follows that $c_{i,j}=0$ for $(i,j)\in\mathcal{P}$ satisfying $i+j = n+2$. Finally, utilizing Theorem~\ref{thm:diamond} as in the case of $\mathcal{P}=[m]\times [n]$, we find that $c_{i,j} = c_{k,\ell}$ for $(i,j),(k,\ell)\in\mathcal{P}$ satisfying $i-j = k-\ell$; that is, all together, we have that 
    \begin{itemize}
    \item $c = - \sum_{i=0}^{n-1} c_{1+i, n-i}$,
    \item $c_{i,j} =0$ for $(i,j)\in\mathcal{P}$ satisfying $i+j = n+2$, and
    \item $c_{i,j} = c_{k,\ell}$ for $(i,j),(k,\ell)\in\mathcal{P}$ satisfying $i-j = k-\ell$.
\end{itemize}
Consequently, it follows that $\dim I_T(\rt{A_n})$ is at most $$|\{(i,j)\in \rt{A_n}~|~i+j=n+1\}|=n=\mathrm{rk}(\mathcal{P})+1,$$ as desired.

\item[$(d)$] Take $\mathcal{P}=\Phi^+(B_n)$ for $n\ge 2$. Applying Lemma~\ref{lem:pchain}, we find that $c = - \sum_{i=0}^{n-1} c_{2n-1-i,1+i}$. Next, note that if $(i,j)\in\mathcal{P}$ satisfies $i+j=2n+1$, then 
    \begin{itemize}
        \item $(i-1,j),(i,j-1)\in\mathcal{P}$ are minimal elements of $\mathcal{P}$ and
        \item $(i,j)$ is the unique element of $\mathcal{P}$ satisfying $(i-1,j),(i,j-1)\precdot (i,j)$. 
    \end{itemize}
    Thus, applying Lemma~\ref{lem:root_zero}, it follows that $c_{i,j}=0$ for $(i,j)\in\mathcal{P}$ satisfying $i+j=2n+1$. Finally, utilizing Theorem~\ref{thm:diamond} as in the case of $\mathcal{P}=[m]\times [n]$, we find that $c_{i,j} = c_{k,\ell}$ for $(i,j),(k,\ell)\in\mathcal{P}$ satisfying $i-j = h-k>0$; that is, all together we have that
    \begin{itemize}
        \item $c = - \sum_{i=0}^{n-1} c_{2n-1-i,1+i}$,
        \item $c_{i,j} = 0$ for $(i,j)\in\mathcal{P}$ satisfying $i+j=2n+1$, and
        \item $c_{i,j} = c_{k,\ell}$ for $(i,j),(k,\ell)\in\mathcal{P}$ satisfying $i-j = h-k>0$.
    \end{itemize}
    Consequently, it follows that $\dim I_T(\rt{B_n})$ is at most $$|\{(i,j)\in \rt{B_n}~|~i+j=2n\}\cup \{k\in\mathbb{Z}~|~(k,k)\in \rt{B_n},~k>n\}|=n+(n-1)=2n-1=\mathrm{rk}(\mathcal{P})+1,$$ as desired.   \qedhere     
\end{enumerate}
\end{proof}

\subsection{Lower bound: $\dim I_T(\mathcal{P})\ge \mathrm{rk}(\mathcal{P})+1$}\label{sec:LB}

In this subsection, we finish the proof of Theorem~\ref{thm:main1} by showing that $\dim I_T(\mathcal{P})\ge \mathrm{rk}(\mathcal{P})+1$ for the posets $\mathcal{P}$ of Theorem~\ref{thm:main1}. To do so, first we identify certain elements of $I_T(\mathcal{P})$ for restricted posets $\mathcal{P}$. Then, focusing on posets $\mathcal{P}$ in Theorem~\ref{thm:main1}, we are able to identify collections of $\mathrm{rk}(\mathcal{P})+1$ linearly independent elements belonging to $I_T(\mathcal{P})$.

\begin{lemma}\label{lem:ucovs}
    Let $\mathcal{P}$ be a poset and $p\in \mathcal{P}$. If there exists unique $(r,q)\in\mathcal{P}^2$ for which $r\precdot p\precdot q$, then $\mathcal{T}_p\in I_T(\mathcal{P})$.
\end{lemma}
\begin{proof}
    It is straightforward to verify that $\mathcal{T}_p=\mathbbm{1}_r-2\mathbbm{1}_p+\mathbbm{1}_q$.
\end{proof}

In Example~\ref{ex:ucovs}, we apply Lemma~\ref{lem:ucovs} to the poset $\rt{B_3}$.

\begin{exmp}\label{ex:ucovs}
    Taking $\mathcal{P}=\rt{B_3}$, illustrated in Figure~\ref{fig:theposets} (d), the element $p=(4,4)$ satisfies the hypotheses of Lemma~\ref{lem:ucovs} with $r=(4,3)$ and $q=(5,4)$. Consequently, applying Lemma~\ref{lem:ucovs}, one finds that $\mathcal{T}_{(4,4)}\in I_T(\mathcal{P})$.
\end{exmp}

\begin{proposition}\label{prop:pchain2}
    Let $\mathcal{P}$ be a poset defined by a connected diagram which is both row and column convex, $M=\mathrm{Min}(\mathcal{P})$, and $$\mathcal{P}_\ell=\{(i,j)\in\mathcal{P}~|~i-j=\ell\}$$ for $\ell\in\mathbb{Z}$. Given $k\in\mathbb{Z}$, suppose that
    \begin{itemize}
        \item[$(i)$] $\mathcal{P}_{k},\mathcal{P}_{k\pm1}\neq \emptyset$ and
        \item[$(ii)$] $(i,j),(i+1,j+1)\in\mathcal{P}_k$ if and only if $(i,j+1)\in\mathcal{P}_{k-1}$ and $(i+1,j)\in\mathcal{P}_{k+1}$.
    \end{itemize}
    Then
    $$f_k=\begin{cases}
        \displaystyle\sum_{\substack{(i,j)\in\mathcal{P}\\ i-j=k}}\mathcal{T}_{(i,j)}, & \mathcal{P}_k\cap M=\emptyset \\
        &\\
        -1+\displaystyle\sum_{\substack{(i,j)\in\mathcal{P}\\ i-j=k}}\mathcal{T}_{(i,j)}, & \mathcal{P}_k\cap M\neq\emptyset
    \end{cases}$$
    belongs to $I_T(\mathcal{P})$.
\end{proposition}

In Example~\ref{ex:pchain2}, we apply Proposition~\ref{prop:pchain2} to the poset $[3]\times [4]$. 

\begin{exmp}\label{ex:pchain2}
    Let $\mathcal{P}=[3]\times[4]$, illustrated in Figure~\ref{fig:theposets} (a). In the notation of Proposition~\ref{prop:pchain2}, we have $$\mathcal{P}_2=\{(3,1)\},\quad \mathcal{P}_1=\{(2,1),(3,2)\},\quad \mathcal{P}_0=\{(1,1),(2,2),(3,3)\},\quad \mathcal{P}_{-1}=\{(1,2),(2,3),(3,4)\},$$ $$\mathcal{P}_{-2}=\{(1,3),(2,4)\},\quad\text{and}\quad\mathcal{P}_{-3}=\{(1,4)\}.$$ One can check $k=1,0,-1,-2$ satisfy the hypotheses of Proposition~\ref{prop:pchain2}. Consequently, applying Proposition~\ref{prop:pchain2}, one finds that $$f_1=\mathcal{T}_{(2,1)}+\mathcal{T}_{(3,2)},~f_0=-1+\mathcal{T}_{(1,1)}+\mathcal{T}_{(2,2)}+\mathcal{T}_{(3,3)},~f_{-1}=\mathcal{T}_{(1,2)}+\mathcal{T}_{(2,3)}+\mathcal{T}_{(3,4)},~f_{-2}=\mathcal{T}_{(1,3)}+\mathcal{T}_{(2,4)}\in I_T(\mathcal{P}).$$
\end{exmp}

\begin{proof}[Proof of Proposition \ref{prop:pchain2}]
    Letting $$\mathcal{Q}_k=\{q\in \mathcal{P}_{k-1}\cup\mathcal{P}_{k+1}~|~\exists~p\in\mathcal{P}_k~\text{such that}~p\precdot q~\text{or}~q\precdot p\}$$
    and $$g_k=-2\sum_{p\in \mathcal{P}_k}\mathbbm{1}_p+\sum_{q\in \mathcal{Q}_k}\mathbbm{1}_q,$$ to establish the result we show that $f_k=g_k$. Assume that $$\mathcal{P}_k=\{p_1\prec p_2\prec\cdots\prec p_n\}.$$ Note that considering (ii), if the poset induced by $\mathcal{Q}_k\cup\mathcal{P}_k$ in $\mathcal{P}$ has a minimal (resp., maximal) element in $\mathcal{Q}_k$, then it is the unique minimal (resp., maximal) element of $\mathcal{Q}_k\cup\mathcal{P}_k$. If such a minimal (resp., maximal) element exists, then we denote it by $m_0$ (resp., $m_n$). Moreover, $$(\mathcal{Q}_k\cap\mathcal{P}_{k-1})\backslash\{m_0,m_n\}=\{r_1\prec r_2\prec\cdots\prec r_{n-1}\}$$ and $$(\mathcal{Q}_k\cap\mathcal{P}_{k+1})\backslash\{m_0,m_n\}=\{q_1\prec q_2\prec\cdots\prec q_{n-1}\}$$ where $p_i\precdot r_i,q_i\precdot q_{i+1}$ for $1\le i<n$. Now, assume that $\mathcal{P}_k\cap M=\emptyset$; the other case following via a similar argument. Note that $\mathcal{P}_k\cap M=\emptyset$ implies that $m_0\in\mathcal{P}$. Take $I\in\mathcal{J}(\mathcal{P})$ and let $S=\mathrm{Max}(\mathcal{P}_I)$, i.e., $S$ denotes the collection of maximal elements in the poset induced by $I$ in $\mathcal{P}$. There are eight cases.
    \begin{itemize}
        \item[$(1)$] If $m_0\notin S$, then, evidently, $f(I)=g(I)=0$.
        \item[$(2)$] If $m_0\in S$, then $$f(I)=\mathcal{T}_{p_1}(I)=1=\mathbbm{1}_{m_o}(I)=g(I).$$
        \item[$(3)$] If $p_i\in S$ for $1\le i<n$, then $$f(I)=\mathcal{T}_{p_i}(I)=-1=-2\mathbbm{1}_{p_i}(I)+\mathbbm{1}_{m_0}(I)=-2\mathbbm{1}_{p_i}(I)+\mathbbm{1}_{m_0}(I)+\sum_{j=1}^{i-1}(-2\mathbbm{1}_{p_j}(I)+\mathbbm{1}_{r_j}(I)+\mathbbm{1}_{q_j}(I))=g(I).$$
        \item[$(4)$] If $r_i\in S$ and $q_i\notin S$, then 
        \begin{align*}
            f(I)=0&=-2\mathbbm{1}_{p_i}(I)+\mathbbm{1}_{m_0}(I)+\mathbbm{1}_{r_i}(I)\\
            &=-2\mathbbm{1}_{p_i}(I)+\mathbbm{1}_{m_0}(I)+\mathbbm{1}_{r_i}(I)+\sum_{j=1}^{i-1}(-2\mathbbm{1}_{p_j}(I)+\mathbbm{1}_{r_j}(I)+\mathbbm{1}_{q_j}(I))\\
            &=g(I).
        \end{align*}
        \item[$(5)$] If $r_i\notin S$ and $q_i\in S$, then $f(I)=g(I)$ via a symmetric argument to that given in $(4)$.
        \item[$(6)$] If $r_i,q_i\in S$, then 
        \begin{align*}
            f(I)=\mathcal{T}_{p_{i+1}}(I)=1=\mathbbm{1}_{m_0}(I)=\mathbbm{1}_{m_0}(I)+\sum_{j=1}^i(-2\mathbbm{1}_{p_j}(I)+\mathbbm{1}_{r_j}(I)+\mathbbm{1}_{q_j}(I))=g(I).
        \end{align*}
        \item[$(7)$] If $p_n\in S$, then 
        \begin{align*}
            f(I)=\mathcal{T}_{p_n}(I)=-1&=-2\mathbbm{1}_{p_n}(I)+\mathbbm{1}_{m_0}(I) \\
            &=-2\mathbbm{1}_{p_n}(I)+\mathbbm{1}_{m_0}(I)+\sum_{j=1}^{n-1}(-2\mathbbm{1}_{p_j}(I)+\mathbbm{1}_{r_j}(I)+\mathbbm{1}_{q_j}(I))\\
            &=g(I).
        \end{align*}
        \item[$(8)$] If $m_n\in S$, then 
        \begin{align*}
            f(I)=0&=-2\mathbbm{1}_{p_n}(I)+\mathbbm{1}_{m_0}(I)+\mathbbm{1}_{m_n}(I)\\
            &=-2\mathbbm{1}_{p_n}(I)+\mathbbm{1}_{m_0}(I)+\mathbbm{1}_{m_n}(I)+\sum_{j=1}^{n-1}(-2\mathbbm{1}_{p_j}(I)+\mathbbm{1}_{r_j}(I)+\mathbbm{1}_{q_j}(I))\\
            &=g(I).
        \end{align*}
    \end{itemize}
    The result follows.
\end{proof}

Utilizing Lemma~\ref{lem:pchain} and Proposition~\ref{prop:pchain2} above, we can now finish the proof of Theorem~\ref{thm:main1}.

\begin{proof}[Proof of Theorem~\ref{thm:main1}]~
In Proposition~\ref{prop:UB}, we showed that $\dim I_T(\mathcal{P})\le \mathrm{rk}(\mathcal{P})+1$ for the posets $\mathcal{P}=[m]\times [n]$, $([n]\times [n])\backslash\mathfrak{S}_2$, $\rt{A_n}$, and $\rt{B_n}$ for $n,m\ge 2$. Thus, it remains to show that $\dim I_T(\mathcal{P})\ge \mathrm{rk}(\mathcal{P})+1$ for such posets $\mathcal{P}$. To do so, we find an appropriately sized linearly independent subset of $I_T(\mathcal{P})$.
\begin{itemize}
    \item[$(a)$] Let $\mathcal{P}=[m]\times [n]$ for $m,n\ge 2$. Applying Lemma~\ref{lem:ucovs} and Proposition~\ref{prop:pchain2}, it follows that $$f_k=\begin{cases}
        \displaystyle\sum_{\substack{(i,j)\in\mathcal{P}\\ i-j=k}}\mathcal{T}_{(i,j)}, & k\neq 0 \\
        &\\
        -1+\displaystyle\sum_{(i,i)\in\mathcal{P}}\mathcal{T}_{(i,i)}, & k=0
    \end{cases}$$ 
    belongs to $I_T(\mathcal{P})$ for $k\in [1-n,m-1]$. Thus, since the $\mathcal{T}_p$ for $p\in\mathcal{P}$ are linearly independent, we have that $\{f_k~|~k\in [1-n,m-1]\}$ is a linearly independent subset of $I_T(\mathcal{P})$. Consequently, $\dim I_T(\mathcal{P})\ge n+m-1=\mathrm{rk}(\mathcal{P})+1$ and the result follows.

    \item[$(b)$] Let $\mathcal{P}=([n]\times [n])\backslash\mathfrak{S}_2$ for $n\ge 2$. Applying Lemma~\ref{lem:ucovs} and Proposition~\ref{prop:pchain2}, it follows that $$f_k=\displaystyle\sum_{\substack{(i,j)\in\mathcal{P}\\ i-j=k}}\mathcal{T}_{(i,j)}\in I_T(\mathcal{P})$$ for $k\in [1-n,0)$. In addition, $h_i=\mathcal{T}_{(i,i)}\in I_T(\mathcal{P})$ for $i\in (1,n)$ as a consequence of Lemma~\ref{lem:ucovs}, and it is straightforward to verify that $h_1=\mathcal{T}_{(1,1)}-1=\mathbbm{1}_{(1,2)}-2\mathbbm{1}_{(1,1)}$ and $h_n=\mathcal{T}_{(n,n)}=\mathbbm{1}_{(n-1,n)}-2\mathbbm{1}_{(n,n)}$ so that $h_1,h_n\in I_T(\mathcal{P})$. Thus, since the $\mathcal{T}_p$ for $p\in\mathcal{P}$ are linearly independent, we have that the $\{f_k~|~k\in (0,n-1]\}\cup \{h_i~|~i\in [1,n]\}$ is a linearly independent subset of $I_T(\mathcal{P})$. Consequently, $\dim I_T(\mathcal{P})\ge 2n-1=\mathrm{rk}(\mathcal{P})+1$ and the result follows.

    \item[$(c)$] Let $\mathcal{P}=\rt{A_n}$ for $n\ge 2$. Applying Proposition~\ref{prop:pchain2}, it follows that $$f_k=-1+\sum_{\substack{(i,j)\in\mathcal{P}\\ i-j=2k+1-n}}\mathcal{T}_{(i,j)}\in I_T(\mathcal{P})$$ for $k\in (0,n-1)$. In addition, it is straightforward to verify that $f_0=-1+\mathcal{T}_{(1,n)}=\mathbbm{1}_{(2,n)}-2\mathbbm{1}_{(1,n)}$ and $f_{n-1}=-1+\mathcal{T}_{(n,1)}=\mathbbm{1}_{(n,2)}-2\mathbbm{1}_{(n,1)}$ so that $f_0,f_{n-1}\in I_T(\mathcal{P})$. Thus, since the $\mathcal{T}_p$ for $p\in\mathcal{P}$ are linearly independent, we have that the $\{f_k~|~k\in [0,n-1]\}$ is a linearly independent subset of $I_T(\mathcal{P})$. Consequently, $\dim I_T(\mathcal{P})\ge n=\mathrm{rk}(\mathcal{P})+1$ and the result follows.

    \item[$(d)$] Let $\mathcal{P}=\rt{B_n}$ for $n\ge 2$. Applying Proposition~\ref{prop:pchain2}, it follows that $$f_k=-1+\sum_{\substack{(i,j)\in\mathcal{P}\\ i-j=-2k}}\mathcal{T}_{(i,j)}\in I_T(\mathcal{P})$$ for $k\in (0,n-1)$. In addition, $h_i=\mathcal{T}_{(i,i)}\in I_T(\mathcal{P})$ for $i\in [n+1,2n-1)$ by Lemma~\ref{lem:ucovs}, and it is straightforward to verify that $f_{n-1}=-1+\mathcal{T}_{(2n+1,1)}=\mathbbm{1}_{(2n+1,2)}-2\mathbbm{1}_{(2n+1,1)}$, $h_n=\mathcal{T}_{(n,n)}=\mathbbm{1}_{(n+1,n)}-2\mathbbm{1}_{(n,n)}$, and $h_{2n-1}=\mathcal{T}_{(2n-1,2n-1)}=\mathbbm{1}_{(2n-1,2n-2)}-2\mathbbm{1}_{(2n-1,2n-1)}$ so that $f_{n-1},h_n,h_{2n-1}\in I_T(\mathcal{P})$. Thus, since the $\mathcal{T}_p$ for $p\in\mathcal{P}$ are linearly independent, we have that the $\{f_k~|~k\in (0,n-1]\}\cup \{h_i~|~i\in [n,2n-1]\}$ is a linearly independent subset of $I_T(\mathcal{P})$. Consequently, $\dim I_T(\mathcal{P})\ge 2n-1=\mathrm{rk}(\mathcal{P})+1$ and the result follows. \qedhere
\end{itemize}
\end{proof}

To end this section, we provide two bases for each of the spaces $I_T(\mathcal{P})$ considered in Theorem~\ref{thm:main1}: one given in terms of the $\mathcal{T}_p$ and the other in terms of the $\mathbbm{1}_p$.

\begin{thm}\label{thm:bases}~
    \begin{itemize}
        \item[$(a)$] If $\mathcal{P}=[n]\times [m]$, then the following collections form bases of $I_T(\mathcal{P})$: \[ \mathcal{B}_1=\left\{\sum_{\substack{(i,j)\in\mathcal{P} \\ i-j=k}} \mathcal{T}_{(i,j)} : k\in [1-n,m-1],~k\neq 0
 \right\}\cup \left\{-1+\sum_{\substack{(i,j)\in\mathcal{P} \\ i-j=0}} \mathcal{T}_{(i,j)}\right\}\] and \[ \mathcal{B}_2=\left\{\sum_{\substack{(i,j)\in\mathcal{P} \\ i-j=k}} \mathbbm{1}_{(i,j)} : k\in [1-n,m-1]
 \right\}.\]
    \item[$(b)$] If $\mathcal{P}=([n]\times [n])\backslash\mathfrak{S}_2$, then the following collections form bases of $I_T(\mathcal{P})$: \[ \mathcal{B}_1=\left\{\sum_{\substack{(i,j)\in\mathcal{P} \\ i-j=k}} \mathcal{T}_{(i,j)} : k\in [1-n,0)
 \right\}\cup \Bigg\{\mathcal{T}_{(k,k)} : k\in (1,n]\Bigg\}\cup \Bigg\{\mathcal{T}_{(1,1)}-1\Bigg\}\] and 
 \begin{align*}
     \mathcal{B}_2=\left\{\sum_{\substack{(i,j)\in\mathcal{P} \\ i-j=k}} \mathbbm{1}_{(i,j)} : k\in [1-n,0)
 \right\}&\cup\Bigg\{2\mathbbm{1}_{(i,i)}-\mathbbm{1}_{(i-1,i)}-\mathbbm{1}_{(i,i+1)}~:~i\in (1,n)\Bigg\}\\
 &\cup\Bigg\{2\mathbbm{1}_{(1,1)}-\mathbbm{1}_{(1,2)},2\mathbbm{1}_{(n,n)}-\mathbbm{1}_{(n-1,n)}\Bigg\}.
 \end{align*}
 
 \item[$(c)$] If $\mathcal{P}=\rt{A_n}$ for $n\ge 2$, then the following collections form bases of $I_T(\mathcal{P})$: \[ \mathcal{B}_1=\left\{ -1 + \sum_{\substack{(i,j)\in\mathcal{P} \\ i-j=2k+1-n}} \mathcal{T}_{(i,j)} : 0 \leq k \leq n-1
 \right\}\] and \[ \mathcal{B}_2=\left\{ \sum_{\substack{(i,j)\in\mathcal{P} \\i-j=2k+1-n}} 2\mathbbm{1}_{(i,j)} - \sum_{\substack{(i,j)\in\mathcal{P} \\i-j=2k+2-n}}\mathbbm{1}_{(i,j)} - \sum_{\substack{(i,j)\in\mathcal{P} \\i-j=2k-n}}\mathbbm{1}_{(i,j)} : 0 \le k \le n-1 
 \right\}.\]

    \item[$(d)$] If $\mathcal{P}=\rt{B_n}$ for $n\ge 2$, then the following collections form bases of $I_T(\mathcal{P})$: \[ \mathcal{B}_1=\left\{ -1 + \sum_{\substack{(i,j)\in\mathcal{P} \\i-j=2k}} \mathcal{T}_{(i,j)} : 0 < k \leq n-1 \right\}\cup \Biggl\{ \mathcal{T}_{(n,n)}-1 \Biggr\} \cup \Biggl\{ \mathcal{T}_{(i,i)} : n < i \leq 2n-1 \Biggr\}\] and 
    $$\mathcal{B}_2=\left\{ \sum_{\substack{(i,j)\in\mathcal{P} \\i-j=2k}} 2\mathbbm{1}_{(i,j)} - \sum_{\substack{(i,j)\in\mathcal{P} \\i-j=2k-1}}\mathbbm{1}_{(i,j)} -  \sum_{\substack{(i,j)\in\mathcal{P} \\i-j=2k+1}}\mathbbm{1}_{(i,j)} : 0 < k \leq n-1 \right\}$$ $$\cup \Bigg\{ 2\mathbbm{1}_{(k,k)}-\mathbbm{1}_{(k,k-1)}-\mathbbm{1}_{(k+1,k)} : n < k < 2n-1   \Bigg\}\cup \Bigg\{ 2\mathbbm{1}_{(2n-1,2n-1)}-\mathbbm{1}_{(2n-1,2n-2)},2\mathbbm{1}_{(n,n)}-\mathbbm{1}_{(n+1,n)}\Bigg\}.$$
    \end{itemize}
\end{thm}

\begin{exmp}
    For $\mathcal{P}=\rt{B_3}$, we illustrate the bases $\mathcal{B}_1$ and $\mathcal{B}_2$ in Figures~\ref{fig:B1} and~\ref{fig:B2}. Each diagram corresponds to an individual basis element with labels of cells $p$ giving the coefficient of $\mathcal{T}_p$ in the case of $\mathcal{B}_1$ and $\mathbbm{1}_p$ in the case of $\mathcal{B}_2$. For $\mathcal{B}_1$, the basis element corresponding to Figure~\ref{fig:B1} (c) requires the addition of the constant $-1$.

    \begin{figure}[H]
        \centering
        $$\begin{tikzpicture}
            \node at (0,0) {\scalebox{0.5}{
\begin{tikzpicture}
    \draw (0,0)--(5,0)--(5,1)--(4,1)--(4,2)--(3,2)--(3,3)--(2,3)--(2,2)--(1,2)--(1,1)--(0,1)--(0,0);
    \draw (1,1)--(4,1);
    \draw (2,2)--(3,2);
    \draw (1,0)--(1,1);
    \draw (2,0)--(2,2);
    \draw (3,0)--(3,2);
    \draw (4,0)--(4,1);
    \node at (0.5,0.5) {\Large $1$};
    \node at (1.5, 0.5) {};
    \node at (2.5, 0.5) {};
    \node at (3.5, 0.5) {};
    \node at (4.5, 0.5) {};
    \node at (1.5, 1.5) {};
    \node at (2.5, 1.5) {};
    \node at (3.5, 1.5) {};
    \node at (2.5, 2.5) {};
    \node at (2.5,-1) {\Huge $(a)$};
\end{tikzpicture}}};
    \node at (3,0) {\scalebox{0.5}{
\begin{tikzpicture}
    \draw (0,0)--(5,0)--(5,1)--(4,1)--(4,2)--(3,2)--(3,3)--(2,3)--(2,2)--(1,2)--(1,1)--(0,1)--(0,0);
    \draw (1,1)--(4,1);
    \draw (2,2)--(3,2);
    \draw (1,0)--(1,1);
    \draw (2,0)--(2,2);
    \draw (3,0)--(3,2);
    \draw (4,0)--(4,1);
    \node at (0.5,0.5) {};
    \node at (1.5, 0.5) {};
    \node at (2.5, 0.5) {\Large 1};
    \node at (3.5, 0.5) {};
    \node at (4.5, 0.5) {};
    \node at (1.5, 1.5) {\Large 1};
    \node at (2.5, 1.5) {};
    \node at (3.5, 1.5) {};
    \node at (2.5, 2.5) {};
    \node at (2.5,-1) {\Huge $(b)$};
\end{tikzpicture}}};
 \node at (6,0) {\scalebox{0.5}{
\begin{tikzpicture}
    \draw (0,0)--(5,0)--(5,1)--(4,1)--(4,2)--(3,2)--(3,3)--(2,3)--(2,2)--(1,2)--(1,1)--(0,1)--(0,0);
    \draw (1,1)--(4,1);
    \draw (2,2)--(3,2);
    \draw (1,0)--(1,1);
    \draw (2,0)--(2,2);
    \draw (3,0)--(3,2);
    \draw (4,0)--(4,1);
    \node at (0.5,0.5) {};
    \node at (1.5, 0.5) {};
    \node at (2.5, 0.5) {};
    \node at (3.5, 0.5) {};
    \node at (4.5, 0.5) {};
    \node at (1.5, 1.5) {};
    \node at (2.5, 1.5) {};
    \node at (3.5, 1.5) {};
    \node at (2.5, 2.5) {\Large 1};
    \node at (2.5,-1) {\Huge $(c)$};
\end{tikzpicture}}};
 \node at (9,0) {\scalebox{0.5}{
\begin{tikzpicture}
    \draw (0,0)--(5,0)--(5,1)--(4,1)--(4,2)--(3,2)--(3,3)--(2,3)--(2,2)--(1,2)--(1,1)--(0,1)--(0,0);
    \draw (1,1)--(4,1);
    \draw (2,2)--(3,2);
    \draw (1,0)--(1,1);
    \draw (2,0)--(2,2);
    \draw (3,0)--(3,2);
    \draw (4,0)--(4,1);
    \node at (0.5,0.5) {};
    \node at (1.5, 0.5) {};
    \node at (2.5, 0.5) {};
    \node at (3.5, 0.5) {};
    \node at (4.5, 0.5) {};
    \node at (1.5, 1.5) {};
    \node at (2.5, 1.5) {};
    \node at (3.5, 1.5) {\Large 1};
    \node at (2.5, 2.5) {};
    \node at (2.5,-1) {\Huge $(d)$};
\end{tikzpicture}}};
 \node at (12,0) {\scalebox{0.5}{
\begin{tikzpicture}
    \draw (0,0)--(5,0)--(5,1)--(4,1)--(4,2)--(3,2)--(3,3)--(2,3)--(2,2)--(1,2)--(1,1)--(0,1)--(0,0);
    \draw (1,1)--(4,1);
    \draw (2,2)--(3,2);
    \draw (1,0)--(1,1);
    \draw (2,0)--(2,2);
    \draw (3,0)--(3,2);
    \draw (4,0)--(4,1);
    \node at (0.5,0.5) {};
    \node at (1.5, 0.5) {};
    \node at (2.5, 0.5) {};
    \node at (3.5, 0.5) {};
    \node at (4.5, 0.5) {\Large 1};
    \node at (1.5, 1.5) {};
    \node at (2.5, 1.5) {};
    \node at (3.5, 1.5) {};
    \node at (2.5, 2.5) {};
    \node at (2.5,-1) {\Huge $(e)$};
\end{tikzpicture}}};
        \end{tikzpicture}$$
        \caption{$\mathcal{B}_1$ for $\rt{B_3}$}
        \label{fig:B1}
    \end{figure}

    \begin{figure}[H]
        \centering
        $$\begin{tikzpicture}
            \node at (0,0) {\scalebox{0.5}{
\begin{tikzpicture}
    \draw (0,0)--(5,0)--(5,1)--(4,1)--(4,2)--(3,2)--(3,3)--(2,3)--(2,2)--(1,2)--(1,1)--(0,1)--(0,0);
    \draw (1,1)--(4,1);
    \draw (2,2)--(3,2);
    \draw (1,0)--(1,1);
    \draw (2,0)--(2,2);
    \draw (3,0)--(3,2);
    \draw (4,0)--(4,1);
    \node at (0.5,0.5) {\Large $2$};
    \node at (1.5, 0.5) {\Large $-1$};
    \node at (2.5, 0.5) {};
    \node at (3.5, 0.5) {};
    \node at (4.5, 0.5) {};
    \node at (1.5, 1.5) {};
    \node at (2.5, 1.5) {};
    \node at (3.5, 1.5) {};
    \node at (2.5, 2.5) {};
    \node at (2.5,-1) {\Huge $(a)$};
\end{tikzpicture}}};
    \node at (3,0) {\scalebox{0.5}{
\begin{tikzpicture}
    \draw (0,0)--(5,0)--(5,1)--(4,1)--(4,2)--(3,2)--(3,3)--(2,3)--(2,2)--(1,2)--(1,1)--(0,1)--(0,0);
    \draw (1,1)--(4,1);
    \draw (2,2)--(3,2);
    \draw (1,0)--(1,1);
    \draw (2,0)--(2,2);
    \draw (3,0)--(3,2);
    \draw (4,0)--(4,1);
    \node at (0.5,0.5) {};
    \node at (1.5, 0.5) {\Large $-1$};
    \node at (2.5, 0.5) {\Large $2$};
    \node at (3.5, 0.5) {\Large $-1$};
    \node at (4.5, 0.5) {};
    \node at (1.5, 1.5) {\Large $2$};
    \node at (2.5, 1.5) {\Large $-1$};
    \node at (3.5, 1.5) {};
    \node at (2.5, 2.5) {};
    \node at (2.5,-1) {\Huge $(b)$};
\end{tikzpicture}}};
 \node at (6,0) {\scalebox{0.5}{
\begin{tikzpicture}
    \draw (0,0)--(5,0)--(5,1)--(4,1)--(4,2)--(3,2)--(3,3)--(2,3)--(2,2)--(1,2)--(1,1)--(0,1)--(0,0);
    \draw (1,1)--(4,1);
    \draw (2,2)--(3,2);
    \draw (1,0)--(1,1);
    \draw (2,0)--(2,2);
    \draw (3,0)--(3,2);
    \draw (4,0)--(4,1);
    \node at (0.5,0.5) {};
    \node at (1.5, 0.5) {};
    \node at (2.5, 0.5) {};
    \node at (3.5, 0.5) {};
    \node at (4.5, 0.5) {};
    \node at (1.5, 1.5) {};
    \node at (2.5, 1.5) {\Large $-1$};
    \node at (3.5, 1.5) {};
    \node at (2.5, 2.5) {\Large $2$};
    \node at (2.5,-1) {\Huge $(c)$};
\end{tikzpicture}}};
 \node at (9,0) {\scalebox{0.5}{
\begin{tikzpicture}
    \draw (0,0)--(5,0)--(5,1)--(4,1)--(4,2)--(3,2)--(3,3)--(2,3)--(2,2)--(1,2)--(1,1)--(0,1)--(0,0);
    \draw (1,1)--(4,1);
    \draw (2,2)--(3,2);
    \draw (1,0)--(1,1);
    \draw (2,0)--(2,2);
    \draw (3,0)--(3,2);
    \draw (4,0)--(4,1);
    \node at (0.5,0.5) {};
    \node at (1.5, 0.5) {};
    \node at (2.5, 0.5) {};
    \node at (3.5, 0.5) {\Large $-1$};
    \node at (4.5, 0.5) {};
    \node at (1.5, 1.5) {};
    \node at (2.5, 1.5) {\Large $-1$};
    \node at (3.5, 1.5) {\Large $2$};
    \node at (2.5, 2.5) {};
    \node at (2.5,-1) {\Huge $(d)$};
\end{tikzpicture}}};
 \node at (12,0) {\scalebox{0.5}{
\begin{tikzpicture}
    \draw (0,0)--(5,0)--(5,1)--(4,1)--(4,2)--(3,2)--(3,3)--(2,3)--(2,2)--(1,2)--(1,1)--(0,1)--(0,0);
    \draw (1,1)--(4,1);
    \draw (2,2)--(3,2);
    \draw (1,0)--(1,1);
    \draw (2,0)--(2,2);
    \draw (3,0)--(3,2);
    \draw (4,0)--(4,1);
    \node at (0.5,0.5) {};
    \node at (1.5, 0.5) {};
    \node at (2.5, 0.5) {};
    \node at (3.5, 0.5) {\Large $-1$};
    \node at (4.5, 0.5) {\Large $2$};
    \node at (1.5, 1.5) {};
    \node at (2.5, 1.5) {};
    \node at (3.5, 1.5) {};
    \node at (2.5, 2.5) {};
    \node at (2.5,-1) {\Huge $(e)$};
\end{tikzpicture}}};
        \end{tikzpicture}$$
        \caption{$\mathcal{B}_2$ for $\rt{B_3}$}
        \label{fig:B2}
    \end{figure}
\end{exmp}

\begin{proof}[Proof of Theorem \ref{thm:bases}]~
\begin{itemize}
    \item[$(a)$] As noted in the proof of Theorem~\ref{thm:main1} (a), $\mathcal{B}_1$ is a linearly independent collection of $n+m-1=\dim I_T(\mathcal{P})$ vectors in $I_T(\mathcal{P})$. Thus, $\mathcal{B}_1$ forms a basis of $I_T(\mathcal{P})$. Now, for $\mathcal{B}_2$, note that considering the proof of Proposition~\ref{prop:pchain2}, we have that $$\mathcal{B}_1=\left\{\sum_{\substack{(i,j)\in\mathcal{P} \\ i-j=k}}-2\mathbbm{1}_{(i,j)}+\sum_{\substack{(i,j)\in\mathcal{P} \\ i-j=k-1}}\mathbbm{1}_{(i,j)}+\sum_{\substack{(i,j)\in\mathcal{P} \\ i-j=k+1}}\mathbbm{1}_{(i,j)} : k\in [1-n,m-1]\right\}.$$ Let $I_n$ denote the $n\times n$ identity matrix and $U_n$ (resp., $L_n$) denote the $n\times n$ matrix with $i,j$-entry equal to 1 if and only if $i-j=-1$ (resp., $i-j=1$) and 0 otherwise. Then the matrix for expressing the elements of $\mathcal{B}_1$ in terms of elements of the form $\sum_{\substack{(i,j)\in\mathcal{P} \\ i-j=k}}\mathbbm{1}_{(i,j)}$ for $k\in [1-n,m-1]$, denoted $M_{n+m-1}$, can be arranged to have the form $$M_{n+m-1}=-2I_{n+m-1}+U_{n+m-1}+L_{n+m-1}.$$ Expanding the determinant of $M_\ell=-2I_\ell+U_\ell+L_\ell$ along the top row for $\ell\ge 2$, it is straightforward to show that $\det M_2=3$, $\det M_3=-4$, and $\det M_\ell=-2\det M_{\ell-1}-\det M_{\ell-2}$ for $\ell>3$. Consequently, one has that $\det M_\ell=(-1)^\ell(\ell+1)\neq 0$ for $\ell\ge 2$. As this means that $M_{n+m-1}$ reduces to the identity matrix for $n,m\ge 2$, it follows that $\mathcal{B}_2$ provides another choice of basis for $I_T(\mathcal{P})$.

    \item[$(b)$] As noted in the proof of Theorem~\ref{thm:main1} (a), $\mathcal{B}_1$ is a linearly independent collection of $2n-1=\dim I_T(\mathcal{P})$ vectors in $I_T(\mathcal{P})$. Thus, $\mathcal{B}_1$ forms a basis of $I_T(\mathcal{P})$. Now, for $\mathcal{B}_2$, let $$S_1=\left\{\sum_{\substack{(i,j)\in\mathcal{P} \\ i-j=k}} T_{(i,j)} : k\in [1-n,0)
 \right\},$$ $$S_2=\left\{\mathcal{T}_{(k,k)} : k\in (1,n)\right\},$$ and $$S_3=\left\{\mathcal{T}_{(1,1)}-1,\mathcal{T}_{(n,n)}\right\}.$$ Note that $\mathcal{B}_1=S_1\cup S_2\cup S_3$. Moreover, considering the proof of Proposition~\ref{prop:pchain2}, we have that 
    $$S_1=\left\{\sum_{\substack{(i,j)\in\mathcal{P} \\ i-j=k}} -2\mathbbm{1}_{(i,j)}+\sum_{\substack{(i,j)\in\mathcal{P} \\ i-j=k-1}} \mathbbm{1}_{(i,j)}+\sum_{\substack{(i,j)\in\mathcal{P} \\ i-j=k+1}} \mathbbm{1}_{(i,j)} : k\in [1-n,0)
 \right\}$$ $$S_2=\left\{-2\mathbbm{1}_{(i,i)}+\mathbbm{1}_{(i-1,i)}+\mathbbm{1}_{(i,i+1)}~|~i\in (1,n)\right\},$$ and $$S_3=\left\{-2\mathbbm{1}_{(1,1)}+\mathbbm{1}_{(1,2)},-2\mathbbm{1}_{(n,n)}+\mathbbm{1}_{(n-1,n)}\right\}.$$ Let $M_n$ denote the $n\times n$ matrix defined in (a) above and $E_n$ denote the $n\times n$ matrix with a 1 in the $n,n$-entry and 0's elsewhere. Similar to as in $(a)$, the matrix for expressing the elements of $S_1$ along with $$\sum_{f\in S_2\cup S_3}f=\sum_{(i,i)\in\mathcal{P}} -2\mathbbm{1}_{(i,i)}+\sum_{\substack{(i,j)\in\mathcal{P} \\ i-j=-1}} 2\mathbbm{1}_{(i,j)}$$ in terms of elements of the form $\sum_{\substack{(i,j)\in\mathcal{P} \\ i-j=k}} \mathbbm{1}_{(i,j)}$ for $k\in [1-n,0)$, can be arranged to have the form $M_n+E_n$ for $n\ge 2$. Also, as in $(a)$, it can be shown that $M_n+E_n$ has a nonzero determinant. In fact, one can show that $\det(M_n+E_n)=(-1)^{n-1}n$ for $n\ge 2$. Consequently, it follows that $\sum_{\substack{(i,j)\in\mathcal{P} \\ i-j=k}} \mathbbm{1}_{(i,j)}\in I_T(P)$ for $k\in [1-n,0)$. Combining these $n-1$ linearly independent elements with the negatives of those of $S_2$ and $S_3$ gives us $\mathcal{B}_2$. As the elements of $\mathcal{B}_2$ are clearly linearly independent, it follows that $\mathcal{B}_2$ provides a basis of $I_T(\mathcal{P})$ as well. 

    \item[$(c)$]  As noted in the proof of Theorem~\ref{thm:main1} (c), $\mathcal{B}_1$ is a linearly independent collection of $n=\dim I_T(\mathcal{P})$ vectors in $I_T(\mathcal{P})$. Thus, $\mathcal{B}_1$ forms a basis of $I_T(\mathcal{P})$. Now, for $\mathcal{B}_2$, note that considering the proof of Proposition~\ref{prop:pchain2}, we have that the elements of $\mathcal{B}_2$ are just the negatives of the elements of $\mathcal{B}_1$. Consequently, $\mathcal{B}_2$ forms a basis of $I_T(\mathcal{P})$ as well.

    \item[$(d)$] Follows as in (c). \qedhere
    
\end{itemize}
\end{proof}

\subsection{Integer Partitions}\label{sec:ip}

In this subsection, we apply the results of Sections~\ref{sec:UB} and~\ref{sec:LB} to a class of posets defined by integer partitions. Specifically, for the poset $\mathcal{P}$ defined by the Ferrer's diagram of an integer partition $\lambda$, we find that $\dim I_T(\mathcal{P})$ can be computed from the border strip of $\lambda$.

Recall that an \textbf{integer partition} is a finite non-increasing sequence of positive integers. Given an integer partition $\lambda=(\lambda_1,\hdots,\lambda_\ell)\in\mathbb{Z}_{>0}^\ell$, the \textbf{Ferrer's diagram} $F(\lambda)$ of $\lambda$ is the diagram consisting of the cells $\{(i,j)~|~1\le i\le \ell,~1\le j\le \lambda_i\}$. The \textbf{border strip} of $F(\lambda)$, denoted $B(\lambda)$, is the collection of all cells $(i,j)\in F(\lambda)$ for which $(i+1,j+1)\notin F(\lambda)$. We refer to $(i,j)\in B(\lambda)$ as a \textbf{corner cell} if $(i+1,j),(i,j+1)\in B(\lambda)$.
To the integer partition $\lambda$, we associate the poset $\mathcal{P}(\lambda)$ defined by $F(\lambda)$, as in  Example~\ref{ex:part}.

\begin{exmp}\label{ex:part}
The diagram $F(\lambda)$ defining the poset $\mathcal{P}=\mathcal{P}(\lambda)$ for $\lambda=(5,2,1,1)$ is illustrated in Figure~\ref{fig:part}, where the cells are labeled by the associated elements of $\mathcal{P}$. Note that $B(\lambda)$ corresponds to the collection of all cells in $F(\lambda)$ except for the one labeled $(1,1)$. Among these cells, the ones corresponding to $(1,2)$ and $(2,1)$ are the corner cells.

\begin{figure}[h]
    \centering
    $$\scalebox{0.7}{\begin{tikzpicture}
        \draw (0,0)--(1,0)--(1,2)--(2,2)--(2,3)--(5,3)--(5,4)--(0,4)--(0,0);
        \draw (0,1)--(1,1);
        \draw (0,2)--(1,2);
        \draw (0,3)--(2,3);
        \draw (1,4)--(1,2);
        \draw (2,4)--(2,3);
        \draw (3,4)--(3,3);
        \draw (4,4)--(4,3);
        \node at (0.5, 0.5) {$(4,1)$};
        \node at (0.5, 1.5) {$(3,1)$};
        \node at (0.5, 2.5) {$(2,1)$};
        \node at (0.5, 3.5) {$(1,1)$};
        \node at (1.5, 2.5) {$(2,2)$};
        \node at (1.5, 3.5) {$(1,2)$};
        \node at (2.5, 3.5) {$(1,3)$};
        \node at (3.5, 3.5) {$(1,4)$};
        \node at (4.5, 3.5) {$(1,5)$};
    \end{tikzpicture}}$$
    \caption{Diagram defining $\mathcal{P}(\lambda)$ for $\lambda=(5,2,1,1)$}
    \label{fig:part}
\end{figure}
    
\end{exmp}

With the notation above, the main result of this section is as follows.

\begin{thm}\label{thm:mainpartition}
    Let $\lambda$ be an integer partition, $\mathcal{P}=\mathcal{P}(\lambda)$, $N$ denote the number of cells in $B(\lambda)$, and $C$ denote the number of corner cells in $B(\lambda)$. Then $\dim I_T(\mathcal{P})=N-C$.
\end{thm}

\begin{exmp}\label{ex:mainpartition}
    For the poset $\mathcal{P}=\mathcal{P}(\lambda)$ of Example~\ref{ex:part}, using the notation of Theorem~\ref{thm:mainpartition}, we have $N=8$ and $C=2$. Thus, applying Theorem~\ref{thm:mainpartition}, we have $\dim I_T(\mathcal{P})=8-2=6$.
\end{exmp}

\begin{rem}
    Note that, unlike for the posets of Theorem~\ref{thm:main1}, given $\mathcal{P}=\mathcal{P}(\lambda)$ for a partition $\lambda$, it is possible to have $\dim I_T(\mathcal{P})\neq \mathrm{rk}(\mathcal{P})+1$. For example, taking $\lambda=(5,2,1,1)$ as in Example~\ref{ex:mainpartition}, we have that $\dim I_T(\mathcal{P})=6\neq 5=\mathrm{rk}(\mathcal{P})+1$.
\end{rem}

For the proof of Theorem~\ref{thm:mainpartition}, we require the following lemma relating the collections of cells in the border strip of a Ferrer's board to cells $(i,j)$ with $i-j=k$ fixed.

\begin{lemma}\label{lem:partitionhelp}
    Let $\lambda$ be an integer partition and $$S=\{k~|~\exists (i,j)\in F(\lambda)~\text{such that}~i-j=k\}.$$ Then $\phi:B(\lambda)\to S$ defined by $(i,j)\mapsto i-j$ is a bijection.
\end{lemma}
\begin{proof}
    To see that $\phi$ is onto, take $v\in S$. By definition, there exists $(i,j)\in F(\lambda)$ such that $i-j=v$. Taking $k\in \mathbb{Z}_{\ge 0}$ maximal such that $(i+k,j+k)\in F(\lambda)$ provides $(i+k,j+k)\in B(\lambda)$ satisfying $i+k-(j+k)=i-j=v$. Thus, $\phi$ is onto. Now, to see that $\phi$ is one-to-one, assume that $(i,j),(k,\ell)\in B(\lambda)$ satisfy $i-j=k-\ell=v\in S$. Moreover, assume that $i>k$. Then there exists $t\in\mathbb{Z}_{> 0}$ such that $i+t=k$ and $j+t=\ell$. Consequently, since $(i,j),(k,\ell)\in F(\lambda)$, it follows from the definition of $F(\lambda)$ that $(i+1,j+1)\in F(\lambda)$, contradicting our assumption that $(i,j)\in B(\lambda)$. Therefore, $\phi$ is one-to-one and the result follows. 
\end{proof}

Lemma~\ref{lem:partitionhelp} in hand, we can now prove Theorem~\ref{thm:mainpartition}.

\begin{proof}[Proof of Theorem~\ref{thm:mainpartition}]
First, we show that $\dim I_T(\mathcal{P})\le N-C$. Let $$f=c+\sum_{(i,j)\in\mathcal{P}}c_{i,j}\mathcal{T}_{(i,j)}\in I_T(\mathcal{P})$$ and $M=\mathrm{Min}(\mathcal{P})$. Applying Lemma~\ref{lem:pchain}, it follows that $c=-\sum_{p\in M}c_p$. Moreover, applying Theorem~\ref{thm:diamond} as in the proof of Theorem~\ref{prop:UB}, we find that $c_{i,j}=c_{k,\ell}$ for all $(i,j),(k,\ell)\in\mathcal{P}$ satisfying $i-j=k-\ell$. Thus, applying Lemma~\ref{lem:partitionhelp}, it follows that
\begin{equation}\label{eq:partition}
    \dim I_T(\mathcal{P})\le |\{k~|~\exists (i,j)\in\mathcal{P}~\text{such that}~i-j=k\}|=N.
\end{equation}
Now, note that if $p=(i,j)\in B(\lambda)$ is a corner cell, then we can apply Lemma~\ref{lem:root_zero} with $q_1=(i+1,j)$ and $q_2=(i,j+1)$ to conclude that $c_{i,j}=0$. Consequently, considering (\ref{eq:partition}), it follows that $$\dim I_T(\mathcal{P})\le N-C.$$

Finally, to see that $\dim I_T(\mathcal{P})\ge N-C$, take $(i,j)\in B(\lambda)$ with $i-j=k$ that is not a corner cell. If $(i-1,j-1)\in F(\lambda)$, then, applying Proposition~\ref{prop:pchain2}, it follows that $$f_k=\begin{cases}
        \displaystyle\sum_{\substack{(i,j)\in\mathcal{P}\\ i-j=k}}\mathcal{T}_{(i,j)}, & \mathcal{P}_k\cap M=\emptyset \\
        &\\
        -1+\displaystyle\sum_{\substack{(i,j)\in\mathcal{P}\\ i-j=k}}\mathcal{T}_{(i,j)}, & otherwise
    \end{cases}$$
    belongs to $I_T(\mathcal{P})$. On the other hand, if $(i-1,j-1)\notin F(\lambda)$, then there are three possibilities:
\begin{itemize}
    \item $(i-1,j),(i+1,j)\in F(\lambda)$ or $(i,j-1),(i,j+1)\in F(\lambda)$, in which case $\mathcal{T}_{(i,j)}\in I_T(\mathcal{P})$ by Lemma~\ref{lem:ucovs};
    \item $(i-1,j)\in F(\lambda)$ and $(i+1,j)\notin F(\lambda)$, in which case $\mathcal{T}_{(i,j)}=-2\mathbbm{1}_{(i,j)}+\mathbbm{1}_{(i-1,j)}\in I_T(\mathcal{P})$; or
    \item $(i,j-1)\in F(\lambda)$ and $(i+1,j)\notin F(\lambda)$, in which case $\mathcal{T}_{(i,j)}=-2\mathbbm{1}_{(i,j)}+\mathbbm{1}_{(i,j-1)}\in I_T(\mathcal{P})$.
\end{itemize}
Regardless of the case, for each such $(i,j)\in B(\lambda)$, we have that $\mathcal{T}_{(i,j)}\in I_T(\mathcal{P})$. Consequently, since the $\mathcal{T}_{(i,j)}$ for $(i,j)\in\mathcal{P}$ are linearly independent, it follows that $\dim I_T(\mathcal{P})\ge N-C$, completing the proof.
\end{proof}

\section{Antichains}\label{sec:antichain}

In this section, we turn our attention to the antichain toggleability spaces $A_T(\mathcal{P})$. Our goal is to establish the antichain versions of Theorems~\ref{thm:main1} and~\ref{thm:mainpartition}, stated below, combined.

\begin{thm}\label{thm:main2}~
\begin{enumerate}
    \item[$(a)$] If $\mathcal{P}=[m]\times [n]$ for $m,n\ge 2$, then $\dim A_T(\mathcal{P})=\mathrm{rk}(\mathcal{P})+1=n+m-1$.
    \item[$(b)$] If $\mathcal{P}=([n]\times[n])\backslash \mathfrak{S}_2$ for $n\ge 2$, then $\dim A_T(\mathcal{P})=\mathrm{rk}(\mathcal{P})+1=2n-1$.
    \item[$(c)$] If $\mathcal{P}=\Phi^+(A_n)$ for $n\ge 2$, then $\dim A_T(\mathcal{P})=\mathrm{rk}(\mathcal{P})+1=n$.
    \item[$(d)$] If $\mathcal{P}=\Phi^+(B_n)$ for $n\ge 2$, then $\dim A_T(\mathcal{P})=\mathrm{rk}(\mathcal{P})+1=2n-1$.
    \item[$(e)$] If $\mathcal{P}=\mathcal{P}(\lambda)$ for an integer partition $\lambda$, $N$ denotes the number of cells in $B(\lambda)$, and $C$ denotes the number of corner cells in $B(\lambda)$, then $\dim A_T(\mathcal{P})=N-C$.
\end{enumerate}
\end{thm}

\noindent
To do so, we first focus on proving Theorem~\ref{thm:dim} below which generalizes Theorem~\ref{thm:main1} (a) and (b), and establishes that the equality $\dim A_T(\mathcal{P})=\dim I_T(\mathcal{P})$ holds for posets defined by simply-connected diagrams with no outward corners. Such posets are exactly the distributive lattices of order ideals of finite, width-two posets. 

\begin{thm}\label{thm:dim}
    Let $\mathcal{P}$ be a poset defined by a simply connected diagram $D$ with no outward corners. Then $\dim A_T(\mathcal{P})=\dim I_T(\mathcal{P})=\rk(\mathcal{P})+1$.
\end{thm}

\noindent
For our approach to Theorem~\ref{thm:dim}, we build upon the technique of rook statistics introduced by Chan, Haddadan, Hopkins, and Moci in~\cite{chan2017expected}, which was later expanded by Hopkins~\cite{hopkins2017cde,hopkins2019cde} and Defant, Hopkins, Poznanovi\'{c}, and Propp~\cite{defant2021homomesy}. In the process of establishing Theorem~\ref{thm:dim}, we find that Theorem~\ref{thm:main2} follows as a consequence of supporting results needed for the proof of Theorem~\ref{thm:dim}.

Now, as a first step towards proving Theorem~\ref{thm:dim}, in Lemma~\ref{lem:uppers} below we establish a relationship between $A_T(\mathcal{P})$ and $I_T(\mathcal{P})$ for certain posets $\mathcal{P}$. Ongoing, for $f,g:\mathcal{J}(\mathcal{P})\to\mathbb{R}$, we write $f\equiv g$ to express that $f-g\in\text{Span}(\mathcal{T}_p~|~p\in\mathcal{P})$.

\begin{lemma}\label{lem:uppers}
    Let $\mathcal{P}$ be a poset for which $\mathbbm{1}_p\equiv \sum_{q\ge p} c_{q,p} \mathcal{T}_q^{-}$ with $c_{p,p}\neq 0$. Then $\dim A_T(\mathcal{P})=\dim I_T(\mathcal{P})$.
\end{lemma}
\begin{proof}
    Define the linear map $f:\text{span}(\mathbbm{1}_p)\to \text{span}(\mathcal{T}_p^{-})$ by $\mathbbm{1}_p\mapsto \sum_{q\ge p} c_{q,p} \mathcal{T}_{q}^{-}$. With respect to the ordering given by a linear extension of $\mathcal{P}$, where recall that a linear extension of a finite poset is an order preserving bijection from $\mathcal{P}\to [|\mathcal{P}|]$, the associated representing matrix of $f$ is upper triangular with diagonal entries equal to $c_{p,p}$. Thus, since $c_{p,p}\neq 0$, it follows that $f$ is invertible. Moreover, since $f$, evidently, preserves $\equiv$--equivalency, we find that $\dim A_T(\mathcal{P}) = \dim I_T(\mathcal{P})$, as desired.
\end{proof}

Next, we define certain combinations of the operators $\mathcal{T}^-_{(k,\ell)}$ and $\mathcal{T}^+_{(k,\ell)}$ for posets $\mathcal{P}$ defined by simply connected diagrams with no outward corners and specified elements $(i,j)\in\mathcal{P}$. 

\begin{definition}
For $D$ a simply connected diagram with no outward corners and $(i,j)\in D$, define the rook statistic \[R^D_{(i,j)} = \sum_{i'\le i,~j'\le j} \mathcal{T}_{(i',j')}^++\sum_{i'\ge i,~j'\ge j} \mathcal{T}_{(i',j')}^--\sum_{\substack{(i'+1,j'),(i',j'+1)\in D \\ i'< i,~j'< j}}\mathcal{T}_{(i',j')}^--\sum_{\substack{(i'-1,j'),(i',j'-1) \in D \\ i'> i,~j'>j}} \mathcal{T}_{(i',j')}^+\] where $\mathcal{T}^{(\pm)}_{(\ell,k)}=0$ if $(\ell,k)\notin D$. In addition, define $\tilde{R}^D_{(i,j)}$ to be the element of $\text{span}(\mathcal{T}_D^-)$ that is $\equiv-$ equivalent to $R^D_{(i,j)}$. See Figure~\ref{fig:Rij} for an illustration of such a rook statistic.
\end{definition}
\begin{figure}[htb]
    \centering
     \scalebox{0.65}{\begin{tikzpicture}
    \filldraw [red] (11.5,0.5) circle (3pt);
    \filldraw [red] (11.5,1.5) circle (3pt);
    \filldraw [red] (11.5,2.5) circle (3pt);
    \filldraw [red] (11.5,3.5) circle (3pt);
    \filldraw [red] (11.5,4.5) circle (3pt);
    \filldraw [red] (11.5,5.5) circle (3pt);
    
    \filldraw [red] (11.5,7.5) circle (3pt);
    \filldraw [red] (11.5,8.5) circle (3pt);
    \filldraw [red] (11.5,9.5) circle (3pt);
    \filldraw [red] (11.5,10.5) circle (3pt);
    \filldraw [red] (11.5,11.5) circle (3pt);
    \filldraw [red] (11.5,12.5) circle (3pt);
    \filldraw [red] (4.5,13.5) circle (3pt);
    \filldraw [red] (4.5,14.5) circle (3pt);
    \filldraw [red] (0.5,11.5) circle (3pt);
    \filldraw [red] (1.5,11.5) circle (3pt);
    \filldraw [red] (2.5,11.5) circle (3pt);
    \filldraw [red] (3.5,11.5) circle (3pt);
    \filldraw [red] (4.5,11.5) circle (3pt);
    \filldraw [red] (5.5,11.5) circle (3pt);
    \filldraw [red] (6.5,9.5) circle (3pt);
    \filldraw [red] (7.5,9.5) circle (3pt);
    \filldraw [red] (8.5,6.5) circle (3pt);
    \filldraw [red] (9.5,6.5) circle (3pt);
    \filldraw [red] (10.5,6.5) circle (3pt);
    \filldraw [red] (12.5,6.5) circle (3pt);
    \filldraw [red] (13.5,6.5) circle (3pt);
    \filldraw [red] (14.5,6.5) circle (3pt);

\draw[step=1.0,black,thin] (0,0) grid (15,15);
\draw[-, ultra thick, blue] (0,15) -- (0,11)--(6,11)--(6,9)--(8,9)--(8,2)--(11,2)--(11,0)--(15,0)--(15,7)--(12,7)--(12,13)--(5,13)--(5,15)--(0,15);
\draw[-, ultra thick, red] (11.5,6.5)--(14.5,6.5);
\draw[-, ultra thick, red] (11.5,6.5)--(8.5,6.5)--(7.5,9.5)--(6.5,9.5)--(5.5,11.5)--(0.5,11.5);
\draw[-, ultra thick, red] (11.5,0.5)--(11.5,12.5)--(4.5,13.5)--(4.5,14.5);
 \filldraw [blue] (11.5,6.5) circle (3pt);
\tpp{0}{15}
\tpp{1}{15}
\tpp{2}{15}
\tpp{3}{15}
\tpp{0}{14}
\tpp{1}{14}
\tpp{2}{14}
\tpp{3}{14}
\tpp{0}{13}
\tpp{1}{13}
\tpp{2}{13}
\tpp{3}{13}
\tpp{4}{13}
\tpp{5}{13}
\tpp{6}{13}
\tpp{7}{13}
\tpp{8}{13}
\tpp{9}{13}
\tpp{10}{13}
\tpp{6}{12}
\tpp{7}{12}
\tpp{8}{12}
\tpp{9}{12}
\tpp{10}{12}
\tpp{6}{11}
\tpp{7}{11}
\tpp{8}{11}
\tpp{9}{11}
\tpp{10}{11}
\tpp{8}{10}
\tpp{9}{10}
\tpp{10}{10}
\tpp{8}{9}
\tpp{9}{9}
\tpp{10}{9}
\tpp{8}{8}
\tpp{9}{8}
\tpp{10}{8}
\tp{0}{12}
\tp{1}{12}
\tp{2}{12}
\tp{3}{12}
\tp{4}{12}
\tp{5}{12}
\tp{6}{10}
\tp{7}{10}
\tp{8}{7}
\tp{9}{7}
\tp{10}{7}
\tp{11}{7}
\tp{11}{8}
\tp{11}{9}
\tp{11}{10}
\tp{11}{11}
\tp{11}{12}
\tp{11}{13}
\tp{4}{14}
\tp{4}{15}

\tpm{12}{6}
\tpm{13}{6}
\tpm{14}{6}
\tpm{12}{5}
\tpm{13}{5}
\tpm{14}{5}
\tpm{12}{4}
\tpm{13}{4}
\tpm{14}{4}
\tpm{12}{3}
\tpm{13}{3}
\tpm{14}{3}
\tpm{12}{2}
\tpm{13}{2}
\tpm{14}{2}
\tpm{12}{1}
\tpm{13}{1}
\tpm{14}{1}
\tm{11}{1}
\tm{11}{2}
\tm{11}{3}
\tm{11}{4}
\tm{11}{5}
\tm{11}{6}
\tm{11}{7}
\tm{12}{7}
\tm{13}{7}
\tm{14}{7}
\end{tikzpicture}}
    \caption{A rook for a shape with no outward corners. The associated reduced rook has a coefficient of 1 for the antichain indicator function for every cell with a red circle in it, and the cell with a blue circle has a coefficient of 2 for the associated antichain indicator function.}
    \label{fig:Rij}
\end{figure}
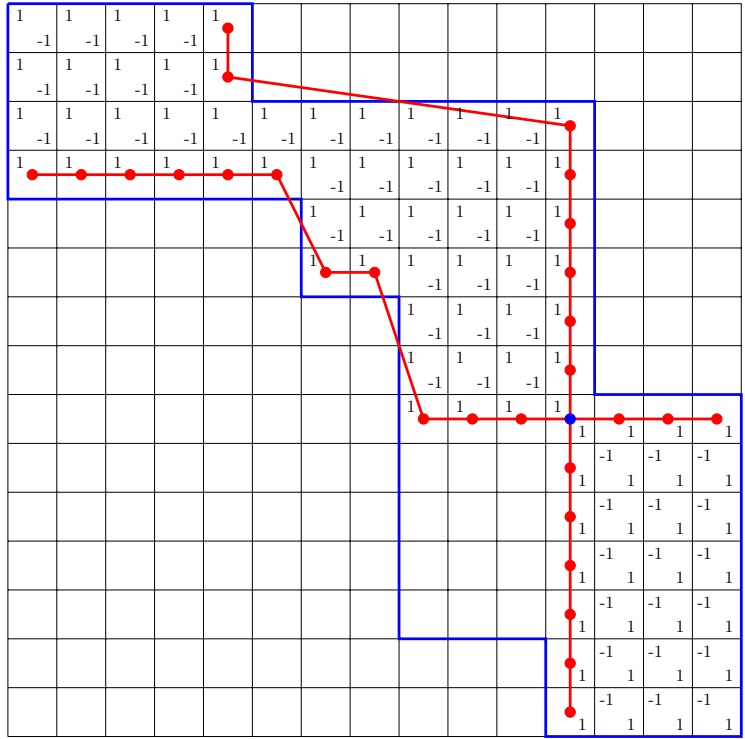

Now, Lemma~\ref{lem:Rij1} and Theorem~\ref{thm:se_boundary_boxes} below establish properties of the functions $R^D_{(i,j)}$ and $\tilde{R}^D_{(i,j)}$. As we will find, a corollary of Lemma~\ref{lem:Rij1} allows us to prove Theorem~\ref{thm:main2}, while all three following results allow us to prove Theorem~\ref{thm:dim}.

\begin{lemma}\label{lem:Rij1}
    For a simply-connected diagram $D$ with no outward corners and $(i,j)\in D$, we have $R^D_{(i,j)}=1$.
\end{lemma}

\begin{proof}
    Define $$S_1=\{\mathcal{T}^+_{(i',j')}~|~i'\le i,~j'\le j,~(i',j')\in D\},\quad S_2=\{\mathcal{T}^-_{(i',j')}~|~i'\ge i,~j'\ge j,~(i',j')\in D\},$$ $$S_3=\{-\mathcal{T}^+_{(i',j')}~|~i'> i,~j'>j,~(i',j'),(i'-1,j'),(i',j'-1) \in D\},$$ $$S_4=\{-\mathcal{T}^-_{(i',j')}~|~i'< i,~j'< j,~(i',j'),(i'+1,j'),(i',j'+1)\in D\},$$ and
    $S=S_1\cup S_2\cup S_3\cup S_4$, so, $R^D_{(i,j)}$ is a sum of the elements of $S$. Let $$S^+_I=\{f\in S~|~f(I)= 1\},\quad\quad S^-_I=\{f\in S~|~f(I)= -1\},\quad\quad\text{and}\quad\quad S_I=S^+_I\cup S^-_I;$$ that is, $S_I=\{f\in S~|~f(I)\neq 0\}$. Note that $R^D_{(i,j)}(I)=\sum_{f\in S_I}f(I)$. 
    
    By assumption, the border of $D$ consists of two paths formed by sequences of unit south and east steps which start at the same point to the top left, end at the same point to the bottom right, and do not intersect at any other points. Since the two paths never intersect outside of their beginning and ending points, one must lie below the other. Consequently, we will refer to the \textit{lower} and \textit{upper} boundary paths, denoted $L$ and $U$, respectively. Now, given an order ideal $I$ of $\mathcal{P}$ and removing the cells corresponding to the elements of $\mathcal{P}\backslash I$, we are left with a collection of cells which lie within $D$ and above a path $P_I$ connecting $L$ and $U$ via a sequence unit steps which can be east or north. For example, see Figure~\ref{fig:ideal}, where the elements contained in $I$ are marked with an $\bigtimes$, $L$ is colored green, $U$ blue, and $P_I$ red.

    \begin{figure}[H]
        \centering
        $$\begin{tikzpicture}[scale = 1]
            \draw[step=1.0,black,thin] (0,0) grid (5,3);
            \draw[-, green, ultra thick] (0,3) -- (0,2)--(1,2)--(1,1)--(2,1)--(2,0)--(5,0);
            \draw[-, blue, ultra thick] (5,0)--(5,2)--(4,2)--(4,3)--(0,3);
            \draw[-, red, ultra thick] (2,1)--(3,1)--(3,2)--(4,2);
         \node at (0.5, 2.5) {$\bigtimes$};
         \node at (1.5, 2.5) {$\bigtimes$};
         \node at (2.5, 2.5) {$\bigtimes$};
         \node at (3.5, 2.5) {$\bigtimes$};
         \node at (1.5, 1.5) {$\bigtimes$};
         \node at (2.5, 1.5) {$\bigtimes$};
        \end{tikzpicture}$$
        \caption{$L$, $U$, and $P_I$ for an ideal $I$}
        \label{fig:ideal}
    \end{figure}
    
    \noindent
    Note that, given the form of the boundary of $D$, the ways in which $P_I$ can intersect $L$ are illustrated in Figure~\ref{fig:PIL}, while the ways in which $P_I$ can intersect $L$ are illustrated in Figure~\ref{fig:PIU}; in both cases, the initial or final unit step of $P_I$ is solid, while any portions of $L$ and $U$ are dashed.
    
    \begin{figure}[H]
        \centering
        \begin{tikzpicture}
            \node at (0,0) {\begin{tikzpicture}
            \draw[dashed] (0,0)--(0.5,0)--(0.5,-0.5);
            \draw (0.5,0)--(0.5,0.5);
        \end{tikzpicture}};
        \node at (0,-1) {$(a)$};

        \node at (2,0) {\begin{tikzpicture}
            \draw[dashed] (0,0)--(0.5,0)--(0.5,-0.5);
            \draw (0.5,0)--(1,0);
        \end{tikzpicture}};
        \node at (2,-1) {$(b)$};

        \node at (4,0) {\begin{tikzpicture}
            \draw[dashed] (0,0)--(0,0.5)--(0,1);
            \draw (0,0.5)--(0.5,0.5);
        \end{tikzpicture}};
        \node at (4,-1) {$(c)$};

         \node at (6,0) {\begin{tikzpicture}
            \draw[dashed] (0,0)--(0.5,0)--(1,0);
            \draw (0.5,0)--(0.5,0.5);
        \end{tikzpicture}};
        \node at (6,-1) {$(d)$};

        \end{tikzpicture}
        \caption{Intersection of $P_I$ with $L$}
        \label{fig:PIL}
    \end{figure}

    \begin{figure}[H]
        \centering
        \begin{tikzpicture}
            \node at (0,0) {\begin{tikzpicture}
            \draw[dashed] (0.5,0)--(0,0)--(0,0.5);
            \draw (-0.5,0)--(0,0);
        \end{tikzpicture}};
        \node at (0,-1) {$(a)$};

        \node at (2,0) {\begin{tikzpicture}
            \draw[dashed] (0.5,0)--(0,0)--(0,0.5);
            \draw (0,-0.5)--(0,0);
        \end{tikzpicture}};
        \node at (2,-1) {$(b)$};

        \node at (4,0) {\begin{tikzpicture}
            \draw[dashed] (0,0)--(0,0.5)--(0,1);
            \draw (0,0.5)--(-0.5,0.5);
        \end{tikzpicture}};
        \node at (4,-1) {$(c)$};

         \node at (6,0) {\begin{tikzpicture}
            \draw[dashed] (0,0)--(0.5,0)--(1,0);
            \draw (0.5,0)--(0.5,-0.5);
        \end{tikzpicture}};
        \node at (6,-1) {$(d)$};

        \end{tikzpicture}
        \caption{Intersection of $P_I$ with $U$}
        \label{fig:PIU}
    \end{figure}

    We break the remainder of the proof into two cases depending on whether or not $(i,j)\in I$. Ongoing, we denote unit north and east steps by $N$ and $E$, respectively.
    \bigskip

    \noindent
    \textbf{Case 1:} $(i,j)\notin I$. In this case, we extend $P_I$ to include any adjacent unit south steps of $L$ starting at the intersection point with $P_I$ as in Figure~\ref{fig:PIL} (b) and (c), now thought of as $N$ steps of $P_I$. Similarly, we extend $P_I$ to include any $E$ steps of $U$ beginning at the intersection point with $P_I$ as in Figure~\ref{fig:PIU} (b) and (d). Note that in this case, elements of $S_I^+$ are of the form $\mathcal{T}^+_{(i',j')}\in S_1$ and are in one-to-one correspondence with $(i',j')\in D$ for which $i'\le i$, $j'\le j$, and $(i',j')$ is bordered by a $N$ followed by a $E$ step of $P_I$, i.e., a valley of $P_I$. Similarly, elements of $S_I^-$ are of the form $-\mathcal{T}^-_{(i',j')}\in S_4$ and are in one-to-one correspondence with $(i',j')\in D$ for which $i'< i$, $j'< j$, and $(i',j')$ is bordered by a $E$ followed by a $N$ step of $P_I$, i.e., a peak of $P_I$. Evidently, $P_I$ alternates between peaks and valleys. Thus, if we can show that within the region $D'$ of $D$ consisting of cells $(i',j')$ with $i'\le i$ and $j'\le j$, $P_I$ starts and ends with a valley, it will follow that $R^D_{(i,j)}(I)=1$, as claimed. For a contradiction, assume that $P_I$ starts with a peak in $D'$. Then this first peak within $D'$ borders a cell $(k,\ell)$ with $k<i$ and $\ell<j$, and all steps preceding this first peak must be east steps. Consequently, $P_I$ must intersect $L$ as in Figure~\ref{fig:PIL} (b) or (c). In either case, our extension of $P_I$ creates a valley which borders a cell $(k',\ell+1)$ with $k'<k<i$ and $\ell+1\le j$, i.e., we are led to a preceding valley of $P_I$ in $D'$, which is a contradiction. A similar argument applies, following north steps leaving, to show that $P_I$ must end on a valley in $D'$ as well. Therefore, the result follows in this case.
    \bigskip

    \noindent
    \textbf{Case 2:} $(i,j)\in I$. In this case, we extend $P_I$ to include any adjacent unit $E$ steps of $L$ ending at the intersection point with $P_I$ as in Figure~\ref{fig:PIL} (a) and (d). Similarly, we extend $P_I$ to include any adjacent unit south steps of $U$ ending at the intersection point with $P_I$ as in Figure~\ref{fig:PIU} (a) and (c), now thought of as $N$ steps of $P_I$. Note that in this case, elements of $S_I^+$ are of the form $\mathcal{T}^-_{(i',j')}\in S_2$ and are in one-to-one correspondence with $(i',j')\in D$ for which $i'\ge i$, $j'\ge j$, and $(i',j')$ is bordered by a $E$ followed by a $N$ step of $P_I$, i.e., a peak of $P_I$. Similarly, elements of $S_I^-$ are of the form $-\mathcal{T}^+_{(i',j')}\in S_3$ and are in one-to-one correspondence with $(i',j')\in D$ for which $i'> i$, $j'> j$, and $(i',j')$ is bordered by a $N$ followed by a $E$ step of $P_I$, i.e., a valley of $P_I$. Arguing as in Case 1, $P_I$ alternates between peaks and valleys, starting and ending with a peak in the region of $D$ consisting of all $(i',j')$ with $i'\ge i$ and $j'\ge j$. Thus, $R^D_{(i,j)}(I)=1$, as claimed, and the result follows in this case.
\end{proof}

\begin{cor}\label{cor:half_rooks}
     Let $\mathcal{P}$ be a poset defined by a simply-connected diagram $D$ where for every cell $(i,j)$ in the diagram $D$, there are no outward corners to the South East of $(i,j)$. Then $\dim A_T(\mathcal{P})=\dim I_T(\mathcal{P})$
\end{cor}
\begin{proof}
Note that in the proof of Lemma~\ref{lem:Rij1}, we showed that $$\mathbbm{1}_{(i,j)}=\sum_{i'\ge i,~j'\ge j} \mathcal{T}_{(i',j')}^--\sum_{\substack{(i'-1,j'),(i',j'-1) \in D \\ i'> i,~j'>j}}\mathcal{T}_{(i',j')}^+.$$ To see this, letting $$g=\sum_{i'\ge i,~j'\ge j} \mathcal{T}_{(i',j')}^--\sum_{\substack{(i'-1,j'),(i',j'-1) \in D \\ i'> i,~j'>j}}\mathcal{T}_{(i',j')}^+,$$ note that in Case 1 in the proof of Lemma~\ref{lem:Rij1} we showed that $g(I)=0$ if $(i,j)\notin I$, while in Case 2 we showed that $g(I)=1$ if $(i,j)\in I$. In fact, the argument given applies more generally to $\mathbbm{1}_{(i,j)}$ for elements $(i,j)$ in posets $\mathcal{P}$ defined by simply-connected diagrams $D$ with no outward corners to the South East of $(i,j)$. Now, setting $$S_1=\{(k,\ell)\in D~|~k>i,~\ell>j,~\text{and}~(k-1,\ell),(k,\ell-1) \in D\}$$ and $$S_2=\{(k,\ell)\in D~|~k\ge i,~\ell\ge j,~\text{and}~(k,\ell)\notin S_1\},$$ note that 
\begin{align*}
    g&=\sum_{(i',j')\in S_2}\mathcal{T}_{(i',j')}^--\sum_{(i',j')\in S_1}(\mathcal{T}_{(i',j')}^+-\mathcal{T}_{(i',j')}^-)=\sum_{(i',j')\in S_2}\mathcal{T}_{(i',j')}^--\sum_{(i',j')\in S_1}\mathcal{T}_{(i',j')}\\
    &\equiv \sum_{(i',j')\in S_2}\mathcal{T}_{(i',j')}^-= \sum_{i'\ge i,~j'\ge j} c_{(i',j'),(i,j)} \mathcal{T}_{(i',j')}^{-}
\end{align*}
with $c_{(i,j),(i,j)}=1\neq 0$. Thus, we can apply Lemma~\ref{lem:uppers} and the result follows.
\end{proof}

As noted above, Theorem~\ref{thm:main2} follows immediately from Corollary~\ref{cor:half_rooks}. Consequently, we have completed the proof of Theorem~\ref{thm:main}. Now, we move to Theorem~\ref{thm:se_boundary_boxes} and the proof of Theorem~\ref{thm:dim}.

\begin{thm}\label{thm:se_boundary_boxes}
    For $\mathcal{P}$ a poset defined by a simply connected diagram $D$ with no outward corners, let $B$ be the set of boxes corresponding to the maximal chain of $\mathcal{P}$ obtained by always going south if possible and east otherwise. Then the set $\{\tilde{R}^D_{(i,j)} : (i,j)\in B\}$ is linearly independent.
\end{thm}

\newcommand{\tmd}[2]{
\node at (#1*1+.75,#2*1-.75) {2};
}
\begin{figure}[ht]
 \begin{center}
    \begin{tikzpicture}[scale = .65]

         \draw[step=1.0,black,thin] (0,0) grid (5,3);
         \draw[-, blue, ultra thick] (0,3) -- (0,2)--(1,2)--(1,1)--(2,1)--(2,0)--(5,0)--(5,2)--(4,2)--(4,3)--(0,3);
        
        \tmd{0}{3}
         \tm{1}{3}
         \tm{2}{3}
         \tm{3}{3}
         \tm{4}{2}
        \tm{1}{2}
         \tm{2}{1}

        \begin{scope}[shift = {(6,0)}]

         \draw[step=1.0,black,thin] (0,0) grid (5,3);
         \draw[-, blue, ultra thick] (0,3) -- (0,2)--(1,2)--(1,1)--(2,1)--(2,0)--(5,0)--(5,2)--(4,2)--(4,3)--(0,3);
          \tm{0}{3}
         \tmd{1}{3}
         \tm{2}{3}
         \tm{3}{3}
         \tm{4}{2}
        \tm{1}{2}
         \tm{2}{1}
        \end{scope}
        \begin{scope}[shift = {(-3,-4)}]

         \draw[step=1.0,black,thin] (0,0) grid (5,3);
         \draw[-, blue, ultra thick] (0,3) -- (0,2)--(1,2)--(1,1)--(2,1)--(2,0)--(5,0)--(5,2)--(4,2)--(4,3)--(0,3);
         \tmd{1}{2}
        \tm{0}{3}
         \tm{1}{3}
         \tm{2}{2}
         \tm{3}{2}
         \tm{4}{2}
         \tm{2}{1}
        \end{scope}
         \begin{scope}[shift = {(3,-4)}]

         \draw[step=1.0,black,thin] (0,0) grid (5,3);
         \draw[-, blue, ultra thick] (0,3) -- (0,2)--(1,2)--(1,1)--(2,1)--(2,0)--(5,0)--(5,2)--(4,2)--(4,3)--(0,3);
         \tm{2}{1}
         \tmd{2}{2}
         \tm{3}{2}
         \tm{4}{2}
        \tm{1}{2}
         \tm{0}{3}
         \tm{2}{3}
        \end{scope}
         \begin{scope}[shift = {(9,-4)}]

         \draw[step=1.0,black,thin] (0,0) grid (5,3);
         \draw[-, blue, ultra thick] (0,3) -- (0,2)--(1,2)--(1,1)--(2,1)--(2,0)--(5,0)--(5,2)--(4,2)--(4,3)--(0,3);
         \tmd{2}{1}
         \tm{2}{2}
         \tm{3}{1}
         \tm{4}{1}
        \tm{1}{2}
         \tm{0}{3}
         \tm{2}{3}
        \end{scope}
        \begin{scope}[shift = {(0,-8)}]

         \draw[step=1.0,black,thin] (0,0) grid (5,3);
         \draw[-, blue, ultra thick] (0,3) -- (0,2)--(1,2)--(1,1)--(2,1)--(2,0)--(5,0)--(5,2)--(4,2)--(4,3)--(0,3);
         \tmd{3}{1}
         \tm{3}{2}
         \tm{3}{3}
         \tm{4}{1}
        \tm{1}{2}
         \tm{0}{3}
         \tm{2}{1}
        \end{scope}
        \begin{scope}[shift = {(6,-8)}]

         \draw[step=1.0,black,thin] (0,0) grid (5,3);
         \draw[-, blue, ultra thick] (0,3) -- (0,2)--(1,2)--(1,1)--(2,1)--(2,0)--(5,0)--(5,2)--(4,2)--(4,3)--(0,3);
         \tmd{4}{1}
         \tm{4}{2}
         \tm{3}{3}
         \tm{3}{1}
        \tm{1}{2}
         \tm{0}{3}
         \tm{2}{1}
        \end{scope}
\end{tikzpicture}
\end{center}
    \caption{An example of the basis of reduced rooks in Theorem~\ref{thm:se_boundary_boxes}}
    \label{fig:se_box_example}
\end{figure}
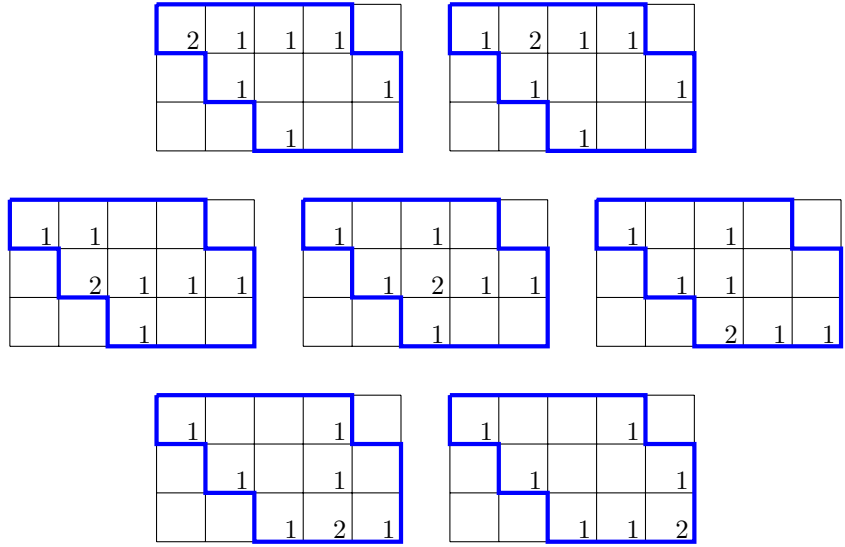

\begin{proof}
    We induct on $|B|$. If $|B| =1$, then $D$ is a single cell and the claim holds trivially. Now suppose that $D$ is a diagram with no outward corners contained minimally in an $m\times n$ rectangle, so that $\rk(\mathcal{P})+1=m+n-1$. Moreover, suppose that for any poset $\mathcal{P}'$ defined by a simply connected diagram $D'$ with no outward corners such that $\rk(\mathcal{P}')+1 < m+n-1$ the claim holds. We may assume that the maximal element of $\mathcal{P}$ is not the only element of the final row of $D$ because we can always pass to the isomorphic poset obtained by swapping the coordinates of our cells in $D$. 

    Let $D'$ be the subdiagram of $D$ obtained by deleting the last column of $D$, $\mathcal{P}'$ be the poset defined by $D'$, and $B'$ be the corresponding southeasternmost maximal chain of $D'$. Note that $\rk(\mathcal{P}')+1=m+n-2$ by construction. By induction, the set $C = \{\tilde{R}^{D'}_{(i,j)}: (i,j)\in B'\}$ considered as reduced rooks in $D'$ are linearly independent. When instead considered as reduced rooks in $D$, the set of reduced rooks $C$ is linearly independent as well, since any such linear combination which will produce 0 in $D$ will also produce 0 in $D'$. Now, suppose for a contradiction that the set $C' = C \cup \{\tilde{R}^D_{(m,n)}\}=\{\tilde{R}^D_{(i,j)} : (i,j)\in B\}$ is not linearly independent as reduced rooks of $D$. Consequently $\tilde{R}^D_{(m,n)} =\sum_{(i,j)\in B'} c_{i,j} \tilde{R}^D_{(i,j)}$. We note that since $\tilde{R}^D_{(i,j)}$ is 1-mesic for all $(i,j)\in D$ by Lemma~\ref{lem:Rij1}, it follows that $\sum_{(i,j)\in B'} c_{i,j}=1$.
    
   For an arbitrary graded poset $\mathcal{Q}$, suppose that $f = \sum_{p\in \mathcal{Q}} f_p\Tout{p}$. If $\mathcal{O}$ is the antichain-rowmotion orbit of $\mathcal{Q}$ containing the empty antichain, then $\sum_{A\in \mathcal{O}} f(A) = \sum_{p\in \mathcal{Q}} f_{p}$ as every element of $\mathcal{Q}$ appears as an element of an antichain in $\mathcal{O}$ exactly once. In the special case where $f$ is a 1-mesic function, $\sum_{A\in \mathcal{O}}f(A) = \rk(\mathcal{Q})+2$. Thus, since $\tilde{R}^D_{(m,n)}$ is 1-mesic, if $\tilde{R}^D_{(m,n)} = \sum_{(i,j)\in D} d_{i,j} \Tout{(i,j)}$, we have that $\sum_{(i,j)\in D} d_{i,j} = \rk(\mathcal{P})+2= m+n$. Additionally, we know by the definition of $\tilde{R}^D_{(m,n)}$ that $d_{i,j} \ge 0$ for all $(i,j)$. Now consider $g=\sum_{(i,j)\in B'} c_{i,j} \tilde{R}^{D'}_{(i,j)}$ as an element of $\text{span}(\Tout{(i,j)}~|~(i,j)\in D')$. It is immediate that $g = \sum_{(i,j)\in D'} d_{i,j} \Tout{(i,j)}$. Since $g$ is 1-mesic, because $\sum_{(i,j)\in B'} c_{i,j}=1$ and each $\tilde{R}^{D'}_{(i,j)}$ is 1-mesic for $(i,j)\in B'$, we have $\sum_{(i,j)\in D'} d_{i,j} = \rk(\mathcal{P}')+2= m+n-1$. But since $d_{m,n} =2$ by the definition of $\tilde{R}^D_{(m,n)}$, we must have that $\sum_{(i,j)\in D} d_{i,j} \ge m+n+1 > m+n$, contradicting the fact that $\tilde{R}_{(m,n)}$ is 1-mesic. Therefore, the set of reduced rooks $C'$ is linearly independent.
\end{proof}

We can now prove Theorem~\ref{thm:dim}.

\begin{proof}[Proof of Theorem \ref{thm:dim}]
By Corollary~\ref{cor:half_rooks}, we have $\dim A_T(\mathcal{P}) = \dim I_T(\mathcal{P}) $. As a consequence of Theorem~\ref{thm:diamond}, we know that $\dim I_T(\mathcal{P})\le m+n-1=\mathrm{rk}(\mathcal{P})+1$ where $D$ is minimally contained in an $m\times n$ rectangle. To see this, mark the leftmost entry in each row and the topmost entry in each column of $D$. Note that there are $m+n-1$ such cells. Then, considering our restrictions on $D$, for an unmarked cell $(i,j)$, it must be the case that $(i-1,j),(i,j-1),(i-1,j-1)\in D$; that is, Theorem~\ref{thm:diamond} can be applied to $(i,j)$ along with these elements. Consequently, elements of $I_T(\mathcal{P})$ can be defined by their coefficients at the marked elements, i.e., $\dim I_T(\mathcal{P})\le m+n-1$. Additionally, considering Lemma~\ref{lem:Rij1}, we have that $\tilde{R}^D_{(i,j)}\in A_T(\mathcal{P})$ for $(i,j)\in D$ so that, applying Theorem~\ref{thm:se_boundary_boxes}, it follows that $m+n-1\le \dim A_T(\mathcal{P})$. The result follows.
\end{proof}
\bibliographystyle{amsplain}
\bibliography{arxivbib}

\providecommand{\bysame}{\leavevmode\hbox to3em{\hrulefill}\thinspace}
\providecommand{\MR}{\relax\ifhmode\unskip\space\fi MR }
\providecommand{\MRhref}[2]{%
  \href{http://www.ams.org/mathscinet-getitem?mr=#1}{#2}
}
\providecommand{\href}[2]{#2}
\begin{thebibliography}{1}

\bibitem{chan2017expected}
Melody Chan, Shahrzad Haddadan, Sam Hopkins, and Luca Moci, \emph{The expected
  jaggedness of order ideals}, Forum of Mathematics, Sigma, vol.~5, Cambridge
  University Press, 2017, p.~e9.

\bibitem{defant2021homomesy}
Colin Defant, Sam Hopkins, Svetlana Poznanovi\'{c}, and James Propp,
  \emph{Homomesy via toggleability statistics}, Comb. Theory \textbf{3} (2023),
  no.~2, Paper No. 14, 61. \MR{4646095}

\bibitem{ElderToggICS}
Jennifer Elder, Nadia Lafreni\`ere, Erin McNicholas, Jessica Striker, and
  Amanda Welch, \emph{Toggling, rowmotion, and homomesy on interval-closed
  sets}, J. Comb. \textbf{15} (2024), no.~4, 479--528. \MR{4835102}

\bibitem{hopkins2017cde}
Sam Hopkins, \emph{The cde property for minuscule lattices}, Journal of
  Combinatorial Theory, Series A \textbf{152} (2017), 45--103.

\bibitem{hopkins2019cde}
\bysame, \emph{The cde property for skew vexillary permutations}, Journal of
  Combinatorial Theory, Series A \textbf{168} (2019), 164--218.

\bibitem{mertin2024toggleability}
Alec Mertin and Svetlana Poznanovi{\'c}, \emph{Toggleability spaces of fences},
  The Electronic Journal of Combinatorics (2024), P4--46.

\bibitem{stanley2012enumerative}
Richard~P. Stanley, \emph{Enumerative combinatorics: Volume 1}, Cambridge Stud.
  Adv. Math., vol.~49, Cambridge University Press, Cambridge, 2012.

\bibitem{striker2015toggle}
Jessica Striker, \emph{The toggle group, homomesy, and the
  {R}azumov-{S}troganov correspondence}, Electron. J. Combin. \textbf{22}
  (2015), no.~2, Paper 2.57, 17. \MR{3367300}

\bibitem{striker2018rowmotion}
\bysame, \emph{Rowmotion and generalized toggle groups}, Discrete Math. Theor.
  Comput. Sci. \textbf{20} (2018), no.~1, Paper No. 17, 26. \MR{3811480}

\end{thebibliography}
\label{sec:biblio}
\end{document}